\documentclass[]{interact}

\usepackage{epstopdf}

\usepackage[numbers,sort&compress]{natbib}
\bibpunct[, ]{[}{]}{,}{n}{,}{,}
\renewcommand\bibfont{\fontsize{10}{12}\selectfont}

\theoremstyle{plain}
\newtheorem{theorem}{Theorem}[section]
\newtheorem{lemma}[theorem]{Lemma}
\newtheorem{assumption}[theorem]{Assumption}
\newtheorem{corollary}[theorem]{Corollary}

\theoremstyle{definition}
\newtheorem{definition}[theorem]{Definition}

\theoremstyle{remark}

\usepackage{graphicx} 
\usepackage{epstopdf}

\usepackage{subfigure}

\usepackage{amsmath}
\usepackage{amssymb}

\usepackage{enumitem}

\newcommand{\dd}{{\mathrm{d}}}
\newcommand{\barr}{\begin{array}}
	\newcommand{\earr}{\end{array}}
\newcommand{\bvec}{ \left[ \!\! \barr{cccccccccccc} }
\newcommand{\evec}{ \earr \!\! \right] }

\newcommand{\LL}{{\mathcal L}}

\newcommand{\R}{{\mathbb{R}}}

\newcommand{\nm}{n_{\mathrm{m}}}
\newcommand{\nx}{n_{\mathrm{x}}}
\newcommand{\nU}{n_{\mathrm{u}}}

\renewcommand{\o}{k}
\newcommand{\pl}{{k+1}}

\newcommand{\eig}{s}

\usepackage{bbold}
\newcommand{\zero}{\mathbb{0}}
\newcommand{\eye}{\mathbb{1}}

\newcommand{\A}{\mathcal{A}}

\usepackage{color}

\newcommand{\F}{\mathcal{F}}
\newcommand{\Jin}{\tilde{J}_{\mathrm{IN}}}
\newcommand{\Jgn}{\tilde{J}_{\mathrm{GN}}}

\usepackage{algorithm}
\usepackage{algpseudocode}
\algrenewcommand\algorithmicrequire{\textbf{Input:}}
\algrenewcommand\algorithmicensure{\textbf{Output:}}

\usepackage{float}
\newfloat{algorithm}{t}{lop}

\usepackage{color}
\definecolor{wheat}{rgb}{0.96,0.87,0.70}

\begin{document}


\title{Adjoint-based SQP Method with Block-wise quasi-Newton Jacobian Updates for Nonlinear Optimal Control}

\author{
	\name{Pedro Hespanhol\textsuperscript{a,b} and Rien Quirynen\textsuperscript{a}}
	\affil{\textsuperscript{a}Control and Dynamical Systems, Mitsubishi Electric Research Laboratories, Cambridge, MA, 02139, USA.
		{\tt\small quirynen@merl.com}}
	\affil{\textsuperscript{b}Department of Industrial Engineering and Operations Research, University of California, Berkeley, CA 94720, USA.}
}

\maketitle

\begin{abstract}
	Nonlinear model predictive control~(NMPC) generally requires the solution of a non-convex optimization problem at each sampling instant under strict timing constraints, based on a set of differential equations that can often be stiff and/or that may include implicit algebraic equations. This paper provides a local convergence analysis for the recently proposed adjoint-based sequential quadratic programming~(SQP) algorithm that is based on a block-structured variant of the two-sided rank-one~(TR1) quasi-Newton update formula to efficiently compute Jacobian matrix approximations in a sparsity preserving fashion. A particularly efficient algorithm implementation is proposed in case an implicit integration scheme is used for discretization of the optimal control problem, in which matrix factorization and matrix-matrix operations can be avoided entirely. The convergence analysis results as well as the computational performance of the proposed optimization algorithm are illustrated for two simulation case studies of nonlinear MPC.
\end{abstract}


\begin{amscode}
	49K15; 49M37; 90C53; 65K05
\end{amscode}


\section{Introduction}

Optimization based control and estimation techniques have attracted an increasing attention over the past decades. They allow a model-based design framework, in which the system dynamics, performance metrics and constraints can directly be taken into account. Receding horizon techniques such as model predictive control~(MPC) and moving horizon estimation~(MHE) have been studied extensively because of their desirable properties~\cite{Mayne2013} and these optimization-based techniques have already been applied in a wide range of applications~\cite{Ferreau2017}. One of the main practical challenges in implementing such an optimization-based predictive control or estimation scheme, lies in the ability to solve the corresponding nonlinear and generally non-convex optimal control problem~(OCP) under strict timing constraints and typically on embedded hardware with limited computational capabilities and available memory. 
	
Let us consider the following continuous-time formulation of the optimal control problem that needs to be solved at each sampling instant
\begin{subequations} \label{eq:C}
	\begin{alignat}{4}
	\underset{x(\cdot),u(\cdot)}{\text{min}} \quad & \int_{0}^{T} \ell(x(t),u(t)) \,\mathrm{d}t &&  \label{C:obj}\\
	\text{s.t.} \quad\;\; & x_0 - \hat{x}_0 \;= \;0, 	 &&	\label{C:initial}\\
	& 0 \; =  \; f(\dot{x}(t),x(t),u(t)), \quad &&  \forall t \in [0,T], \label{C:path}\\
	& p(x(t),u(t)) \leq 0, \quad && \forall t \in [0,T], \label{C:set} 
	\end{alignat}
\end{subequations}
where $T$ denotes the control horizon length, $x(t) \in \mathbb{R}^{\nx}$ denotes the differential states and $u(t) \in \mathbb{R}^{\nU}$ are the control inputs. The function $\ell(\cdot)$ defines the stage cost and the nonlinear system dynamics are formulated as an implicit system of ordinary differential equations~(ODE) in~\eqref{C:path}, which could additionally be extended with implicit algebraic equations. A common assumption is that the resulting system of differential-algebraic equations~(DAE) is of index $1$~\cite{Bock2005}. The optimization problem is parametric, since it depends on the state estimate $\hat{x}_0$ at the current sampling instant, through the initial value condition in~\eqref{C:initial}. The path constraints are defined by the function $p(\cdot)$ in Eq.~\eqref{C:set} and, for simplicity of notation, they are further assumed to be affine. Note that a similar problem as in~\eqref{eq:C} needs to be solved for optimization-based state and parameter estimation, without the given initial state value.

In direct optimal control methods, one forms a discrete-time approximation of the continuous-time OCP in~\eqref{eq:C} based on an appropriate parameterization of the state and control trajectories over the time horizon $t \in [0,T]$, resulting in a tractable nonlinear program~(NLP) that needs to be solved. Popular examples of this approach include the direct multiple shooting method~\cite{bock1984multiple} and direct collocation~\cite{Betts2010,Biegler1984}. Note that these techniques often need to rely on implicit integration methods in order to deal with stiff and/or implicit systems of differential or differential-algebraic equations~\cite{Quirynen2017}. The resulting constrained optimization problem can be handled by standard Newton-type algorithms such as interior point methods~\cite{Waechter2006} and sequential quadratic programming~(SQP)~\cite{Boggs1995} techniques for nonlinear optimization~\cite{nocedal2006numerical}. 

Quasi-Newton optimization methods are generally popular for solving such a constrained NLP. They result in computationally efficient Newton-type methods that solve the first order necessary conditions of optimality, i.e., the Karush-Kuhn-Tucker~(KKT) conditions, without evaluating the complete Hessian of the Lagrangian and/or even without evaluating the Jacobian of the constraints~\cite{nocedal2006numerical}. Instead, quasi-Newton methods are based on low-rank update formulas for the Hessian and Jacobian matrix approximations~\cite{Dennis1977}. Popular examples of this approach include the Broyden-Fletcher-Goldfarb-Shanno~(BFGS)~\cite{Broyden1967} and the symmetric rank-one~(SR1) update formula~\cite{conn1991convergence} for approximating the Hessian of the Lagrangian. Similarly, quasi-Newton methods can be used for approximating Jacobian matrices, e.g., of the constraint functions, such as the good and bad Broyden methods~\cite{broyden2000discovery} as well as the more recently proposed two-sided rank-one~(TR1) update formula~\cite{griewank2002constrained}.

For the purpose of real-time predictive control and estimation, continuation-based online algorithms have been proposed that aim at further reducing the computational effort by exploiting the fact that a sequence of closely related parametric optimization problems is solved~\cite{Bock2005,Diehl2009c}. One popular technique consists of the real-time iteration~(RTI) algorithm that performs a single SQP iteration per time step, in combination with a sufficiently high sampling rate and a prediction-based warm starting in order to allow for closed-loop stability of the system~\cite{Diehl2005}. The RTI algorithm can be implemented efficiently based on (fixed-step) integration schemes with tailored sensitivity propagation for discretization and linearization of the system dynamics~\cite{Quirynen2017} in combination with structure-exploiting quadratic programming solvers~\cite{Ferreau2017}. In addition, a lifted algorithm implementation has been proposed in~\cite{quirynen2017lifted} to directly embed the iterative procedure of implicit integration schemes, e.g., collocation methods, within a Newton-type optimization framework for optimal control.

Unlike standard inequality constrained optimization, nonlinear optimal control problems typically result in a particular sparsity structure in the Hessian of the Lagrangian and in the Jacobian matrix for the equality constraints. In direct optimal control methods, the objective function is typically separable resulting in a block-diagonal Hessian matrix. This property has been exploited in partitioned quasi-Newton methods that approximate and update each of the Hessian block matrices separately, as proposed and studied in~\cite{Griewank1982,Griewank1982a,janka2016sr1}. On the other hand, the Jacobian matrix corresponding to the discretized system dynamics has a block bidiagonal sparsity structure, because of the stage wise coupling of the optimization variables at subsequent time steps of the control horizon. For this purpose, the present article analyzes a novel tailored quasi-Newton method for optimal control using a partitioned or block-structured TR1-based Jacobian update formula. This adjoint-based SQP method for nonlinear optimal control, based on a Gauss-Newton Hessian approximation in combination with inexact Jacobian matrices, was proposed recently in~\cite{Hespanhol2018}.

\subsection{Contributions and outline}
This paper provides a complete presentation of the block-TR1 based SQP method for nonlinear optimal control, including a detailed discussion of the lifted collocation type implementation, extending earlier work of the same authors in~\cite{Hespanhol2018}. Unlike the latter publication, a convergence analysis of this novel quasi-Newton type optimization algorithm is provided. More specifically, we prove convergence of the block-structured quasi-Newton Jacobian approximations to the exact Jacobian matrix within the null space of the active inequality constraints. Based on this result, under mild conditions, convergence of the overall inexact SQP method can be guaranteed. Locally linear or superlinear convergence rates can be shown, respectively, when using a Gauss-Newton or quasi-Newton based Hessian approximation scheme. In addition, it is shown how this convergence analysis extends to our lifted collocation implementation that avoids any matrix factorization or matrix-matrix operations. These convergence analysis results as well as the computational performance of the optimization algorithms are illustrated numerically for two simulation case studies of nonlinear MPC.


	
The paper is organized as follows. Section~\ref{sec:DOC} briefly introduces the direct multiple shooting based OCP problem formulation as well as the proposed adjoint-based inexact SQP method that is based on block-wise TR1 Jacobian updates. Section~\ref{sec:CONV} presents the detailed convergence analysis for the optimization method and contains the main theoretical results of the present paper. A numerically efficient implementation of the block-TR1 update formula in combination with a lifted Newton-type method for direct optimal control with implicit integration schemes such as, e.g., collocation methods, is then proposed and analyzed in Section~\ref{sec:TRD}. Finally, Section~\ref{sec:caseStudies} presents numerical results of the NMPC case studies and Section~\ref{sec:concl} concludes the paper.

\section{Block-wise TR1 based Sequential Quadratic Programming} \label{sec:DOC}

A popular approach for direct optimal control is based on direct multiple shooting~\cite{bock1984multiple} that performs a time discretization, based on a numerical integration scheme~\cite{Hairer1991} to solve the following initial value problem 
\begin{equation}
0 = f(\dot{x}(\tau),x(\tau),u(\tau)), \quad \tau \in [t_i,t_{i+1}], \quad x(t_i) = x_i, \label{eq:IVP}
\end{equation}
on each of $N$ shooting intervals that are defined by a grid of consecutive time points $t_i$ for $i = 0, \ldots, N$. For the sake of simplicity, we consider here an equidistant grid over the control horizon, i.e., $t_{i+1}-t_i = \frac{T}{N}$, and a piecewise constant control parametrization $u(\tau) = u_i$ for $\tau \in [t_i,t_{i+1})$ in~\eqref{eq:IVP}.
An explicit fixed-step integration scheme defines the discrete-time system dynamics $x_{i+1} = F_{i}(x_{i},u_{i})$ for the shooting interval $[t_i,t_{i+1}]$. For example, this can correspond to the popular Runge-Kutta method of order $4$~(RK4) as defined in~\cite{Hairer1991}. Based on the explicit discretization scheme, the resulting block-structured optimal control problem reads as
\begin{subequations} \label{MS-OCP}
	\begin{alignat}{5}
	\underset{X,\,U}{\text{min}} \quad &\sum_{i=0}^{N-1} l_i(x_i,u_i) + l_N(x_N)&& \label{MS:obj}\\
	\text{s.t.} \quad\; & \hat{x}_0 \;= \;x_0,  	\label{MS:initial}\\
	& F_{i}(x_{i},u_{i}) \;=\; x_{i+1}  , \quad &&i = 0, \ldots, N-1, \label{MS:Kstep}\\
	& P_i\, w_i \leq p_i, \quad &&i = 0, \ldots, N, \label{MS:path}
	\end{alignat}
\end{subequations}
where the affine path constraints~\eqref{MS:path} have been imposed on each of the shooting nodes and the compact notation $w_i := (x_i, u_i)$ for $i = 0, \ldots, N-1$ and $w_N := x_N$ is defined. Note that the optimization variables for the problem in~\eqref{MS-OCP} are directly the state $X=[x_0^\top,\ldots, x_N^\top]^\top$ and control trajectory $U=[u_0^\top,\ldots, u_{N-1}^\top]^\top$. 

\subsection{SQP algorithm with inexact Jacobians}

For a local minimum $w^*$ of the NLP in~\eqref{MS-OCP}, for which the linear independence constraint qualification~(LICQ) holds, there must exist a unique set of multiplier values $\lambda^*$, $\mu^{*}$ such that the following Karush-Kuhn-Tucker~(KKT) conditions are satisfied
\begin{subequations} \label{KKT-RK4}
	\begin{alignat}{5}
	\nabla_{w}\mathcal{L}(w^{*}, \lambda^*) + P^{T}\mu^{*} &= 0 && \label{KKT-RK4:DF1}\\
	F(w^{*}) &= X^{*} && \label{KKT-RK4:PF1}\\
	P\, w^{*} &\leq p && \label{KKT-RK4:PF2}\\
	\mu^{*} &\geq 0 && \label{KKT-RK4:DF2}\\
	\mu^{*}_j (P w^{*} - p)_j &= 0, \quad j = 1, \ldots, n_{\mathrm{p}}, \label{KKT-RK4:CS}
	\end{alignat}
\end{subequations}
where $F(\cdot)$ and $P$ are appropriate block-wise concatenations of the equality and inequality constraints, respectively, in~\eqref{MS:Kstep} and~\eqref{MS:path} and $n_{\mathrm{p}}$ denotes the total number of inequality constraints. Here, we also lumped the initial condition constraint as part of the matrix $P$ since we can represent a linear equality constraint as two linear inequality constraints. 
Lastly, $\mathcal{L}(w, \lambda)$ denotes the `truncated Lagrangian', omitting inequality constraints, and is therefore given by
\begin{equation}
\mathcal{L}(w, \lambda) = \sum_{i=0}^{N-1} \left( l_i(x_i,u_i) + \lambda_{i}^\top  (F_i(w_i) - x_{i+1}) \right) + l_N(x_N).
\end{equation}
Given the set of indices $\A$ for the inequality constraints that are active at the local minimum, the KKT system reduces to a nonlinear system of equations that can be solved directly by a Newton-type method.
In particular, we are interested in a quasi-Newton algorithm where we will approximate $\nabla_{ww}^{2}\mathcal{L}(w^k, \lambda^k)$ by a matrix $H^k$ and $\frac{\partial F}{\partial w}(w^k)$ by a matrix $A^k$. Namely, we solve the following linearized system
\begin{equation} \label{QP-KKT}
\begin{bmatrix}
H^k & A^{k^\top} - E^\top & P_A^\top \\
A^{k} - E & \text{ } & \text{ } \\
P_A & \text{} & \text{} 
\end{bmatrix} 
\begin{bmatrix}
\Delta w^k  \\ 
\Delta \lambda^{k} \\
\Delta \mu^{k}_{A}
\end{bmatrix} =
- \begin{bmatrix}
g(w^k,\lambda^k)  \\ 
F(w^k) - X^k \\
P_{A} \, w^k - p_{A}
\end{bmatrix},	
\end{equation}
where $g(w^{k},\lambda^k) = \nabla_{w} \mathcal{L}(w^k,\lambda^k) + \mu^{k^\top}_{A}(P_{A} \, w^k - p_{A})$ at each Newton-type iteration $k$. Note that the matrix $P_{A}$ is defined as the part of $P$ that corresponds to the inequality constraints~\eqref{KKT-RK4:PF2} in the active set $\A$, and $E$ denotes the constant matrix corresponding to the right-hand side of the equality constraints in~\eqref{KKT-RK4:PF1}.

In order to efficiently solve the inequality constrained OCP in~\eqref{MS-OCP}, let us consider the adjoint-based SQP algorithm with Gauss-Newton type Hessian approximation and inexact Jacobian information as introduced originally in~\cite{Bock2005,Wirsching2006} for fast nonlinear MPC. Each SQP iteration solves a convex QP subproblem
\begin{subequations} \label{SQP-MS}
	\begin{alignat}{4}
	\underset{\Delta W}{\text{min}} \quad &\sum_{i=0}^{N} \frac{1}{2} \Delta w_{i}^\top H_{i}^k\, \Delta w_{i} \,+\, h_{i}^{k^\top} \Delta w_{i} &&\label{SQP-MS:obj} \\
	\text{s.t.} \quad & \Delta x_0 \;= \; \hat{x}_0 - x_0^k,  	\label{SQP-MS:initial}\\
	& a_{i}^k + A_{i}^k\, \Delta w_{i} \;= \;\Delta x_{i+1}, \quad &&i = 0,\ldots, N-1, \label{SQP-MS:Kstep}  \\
	& P_i\, \Delta w_i \leq p_{i}^k, \qquad \quad\;\qquad &&i = 0,\ldots, N, \label{SQP-MS:path}
	\end{alignat}
\end{subequations}
where notation $\Delta W=[\Delta w_0^\top,\ldots, \Delta w_N^\top]^\top$ is used to denote the deviation variables $\Delta w_i := w_i - w_i^k$, given the current solution guess $X^k$, $U^k$ for the state and control trajectories at iteration $k$ of the adjoint-based SQP method. The function $p(\cdot)$ that defines the path constraint~\eqref{C:set} was assumed to be affine and $p_{i}^k := p_i - P_i\, w_i^k$.
Note that tracking formulations for nonlinear MPC typically include a stage cost that is defined by a (nonlinear) least squares term $l_{i}(x_{i},u_{i}) \;= \; \frac{1}{2}\Vert R(x_{i},u_{i}) \Vert_2^{2}$ for $i = 0, \ldots, N$. The generalized Gauss-Newton~(GGN) method from~\cite{bock1983recent} uses the block-structured Hessian approximation $H_i^k := \nabla R(w_i^k) \nabla R(w_i^k)^\top \approx \nabla_{w_i w_i}^2 \LL(\cdot)$. 

The matrix $A_{i}^k \approx \frac{\partial F_i}{\partial w_i}(w_i^k)$ denotes the Jacobian approximation and $a_i^k := F_i(w_i^k) - x_{i+1}^k$ for the discrete-time system dynamics in Eq.~\eqref{SQP-MS:Kstep}. For real-time NMPC, such a Jacobian approximation can be obtained by reusing information from a previous NLP solution~\cite{Bock2005,Wirsching2006}. The gradient term in the objective~\eqref{SQP-MS:obj} reads as
\begin{equation}
h_i^k := \nabla_{w_i} l(w_i^k) + \left(\frac{\partial F_i}{\partial w_i}(w_i^k) - A_{i}^k\right)^\top \lambda_i^k,  \label{eq:grad-MS}
\end{equation}
for $i= 0,\ldots, N-1$, in which $\lambda_i^k$ denotes the current value of the Lagrange multipliers for the nonlinear continuity constraints in~\eqref{MS:Kstep}.
Note that the linearized KKT conditions in~\eqref{QP-KKT} correspond to the KKT optimality conditions for the QP in~\eqref{SQP-MS}, for a fixed active set $\A$. In addition, each QP subproblem is convex because $H^k \succeq 0$, e.g., for the Gauss-Newton Hessian approximation. 
A full-step inexact SQP method will sequentially solve each QP subproblem~\eqref{SQP-MS} and perform the following updates:
\begin{equation}
w^{k+1} = w^{k} + \Delta w^{k} \;\text{ and }\; \lambda^{k+1} = \lambda^k + \Delta \lambda^{k} = \lambda^{k+1}_{QP},
\end{equation}
where $\lambda^{k+1}_{QP}$ denote the Lagrange multiplier values for Eq.~\eqref{SQP-MS:Kstep} at the QP solution.
We do not need to perform explicit updates for the Lagrange multipliers associated with the inequality constraints, because they are assumed to be affine, hence not impacting any computation on the QP formulation in~\eqref{SQP-MS}.

\subsection{Dynamic block-wise TR1 Jacobian updates}
At each SQP iteration, we perform the block-wise two-sided rank-one~(TR1) Jacobian update, as proposed recently in~\cite{Hespanhol2018}. Following the work in~\cite{griewank2002constrained}, given current Jacobian approximations $A^{\o}_{i}$ for $i = 0, \ldots, N-1$, we would like that each updated approximation matrix $A^{\pl}_{i}$ satisfies the following two secant conditions
\begin{equation}
\begin{aligned}
\text{Adjoint Condition (AC):} \quad \sigma_{i}^{\o^\top} A^{\pl}_{i} &= \gamma_{i}^{\o^\top}  \\
\text{Forward Condition (FC):} \quad\; A^{\pl}_{i}s_i^\o &= y_i^\o,
\end{aligned} \label{eq:conds}
\end{equation}
where we define the adjoint vector $\gamma_{i}^\o = \frac{\partial F_i}{\partial w_i}(w^{\pl}_i)^\top \sigma_{i}^\o$, given $\sigma_{i}^{\o^\top} = (\lambda^{\pl}_{i} - \lambda^{\o}_{i})^{\top}$, and the difference in function evaluations $y_i^\o = F(w^{\pl}_{i}) - F(w^{\o}_{i})$. Note that $\lambda^{\pl}_{i}$ and $\lambda^{\o}_{i}$, respectively, denote the new and old Lagrange multipliers for the linearized equality constraints in Eq.~\eqref{SQP-MS:Kstep}. Similarly, $w_i^{\o} := (x_i^{\o},u_i^{\o})$ and $w_i^{\pl} := w_i^{\o} + \Delta w_i^{\o}$ denote, respectively, the old and new primal variables, such that $s_i^\o := w_i^{\pl} - w_i^{\o}$. Note that the gradient $\gamma_{i}^\o = \frac{\partial F_i}{\partial w_i}(w^{\pl}_i)^\top \sigma_{i}^\o$ can be computed efficiently using the backward or adjoint mode of algorithmic differentiation~(AD), e.g., see~\cite{Griewank2000}.

The proposed block-wise TR1 update formula then reads as follows
\begin{equation}
A^{\pl}_{i} = A^{\o}_{i} + \alpha_i^\o \left(y_i^\o - A^{\o}_{i}s_{i}^\o\right)\left(\gamma_i^{\o^\top} - \sigma_{i}^{\o^\top} A^{\o}_{i} \right), \label{eq:BTR1}
\end{equation}
for $i = 0, \ldots, N-1$ and where $\alpha_i^\o$ is a scalar that will be defined further. Aside from the case where the function $F(\cdot)$ is affine, the two conditions in Eq.~\eqref{eq:conds} are not consistent with each other and they can therefore generally not both be satisfied by the updated matrix $A^{\pl}_{i}$ at each iteration. Thus, similar to the standard TR1 update in~\cite{griewank2002constrained}, the block-wise update will only be able to satisfy one or the other. In the adjoint variant of the update, the scaling value is defined as
\begin{equation}
\alpha_{\mathrm{A},i}^\o = \frac{1}{\sigma_i^{\o^\top} (y_{i}^\o - A^{\o}_{i}s_i^\o)}, \label{eq:alpha1}
\end{equation}
such that the adjoint condition in~\eqref{eq:conds} is satisfied exactly and the forward condition holds up to some accuracy. Similarly, this value reads as follows for the forward variant
\begin{equation}
\alpha_{\mathrm{F},i}^\o = \frac{1}{(\gamma_{i}^{\o^\top} - \sigma_i^{\o^\top} A^{\o}_{i}) \, s_i^\o}, \label{eq:alpha2}
\end{equation}
where the forward condition is satisfied exactly. It is interesting to note that, since we apply the block-wise TR1 update from~\eqref{eq:BTR1} for each shooting interval $i = 0, \ldots, N-1$, the resulting update for the complete constraint Jacobian matrix of the QP in~\eqref{SQP-MS} corresponds to a rank-$N$ update.

As in~\cite{griewank2002constrained}, we impose a skipping condition in order to avoid a potential blow-up of the block-wise TR1 update when the denominator of the scaling factor becomes small or even zero. For our purposes, the skipping condition itself depends on the type of formula that is used. We update the block matrix $A_i^\o$ only if the following holds
\begin{equation} \label{skip1}
\left| (\gamma_i^{\o^\top} - \sigma_{i}^{\o^\top} A_i^\o) s_i^\o \right| \geq c_1 \left\Vert \sigma_i^\o \right\Vert \, \left\Vert y_i^\o - A_i^\o s_{i}^\o\right\Vert,
\end{equation}
with $c_1 \in (0,1)$ if $\alpha_i^\o = \alpha_{\mathrm{F},i}^\o$ in the forward TR1 update, and
\begin{equation} \label{skip2}
\left| \sigma_{i}^{\o^\top}(y_i^\o - A_i^\o s_{i}^\o) \right| \geq c_1 \left\Vert s_i^\o \right\Vert \, \left\Vert \gamma_i^\o - A_i^{\o^\top} \sigma_{i}^\o \right\Vert,
\end{equation}
with $c_1 \in (0,1)$ if $\alpha_i^\o = \alpha_{\mathrm{A},i}^\o$ in the adjoint TR1 update. 
In addition, in order to consistently choose either the forward or adjoint Jacobian update formula, we propose a more dynamic variant of the algorithm that picks either $\alpha_{\mathrm{F},i}^\o$ or $\alpha_{\mathrm{A},i}^\o$ for each block matrix at any given iteration. It may not be clear what is the best approach to select which type of update is to be executed for each block matrix at a given iteration. However, in the next section, we prove the local convergence properties of the algorithm under any arbitrary sequence of updates that satisfy the skipping conditions in~\eqref{skip1} and~\eqref{skip2} for each block $i$ at every iteration $k$.

 \begin{algorithm}[h]
    	\caption{One iteration of SQP method with block-wise TR1 Jacobian updates.}
    	\label{alg:block_TR1}
    	\begin{algorithmic}[1]
    		\Require $w^{\o}_{i} = (x^{\o}_{i},u^{\o}_{i})$, $\lambda_{i}^{\o}$ and $A^{\o}_{i}$ for $i = 0, \ldots, N-1$.
    		\Statex \texttt{Problem linearization and QP preparation}
    		\State Formulate the QP in~\eqref{SQP-MS} with Jacobian matrices $A^{\o}_{i}$, Gauss-Newton Hessian approximations $H_i^{\o}$ and vectors $a_{i}^{\o}$, $p_{i}^{\o}$ and $h_{i}^{\o}$ in~\eqref{eq:grad-MS} for $i = 0, \ldots, N-1$. \vspace{1mm}		
    		\Statex \texttt{Computation of Newton-type step direction}
    		\State Solve the QP subproblem in Eq.~\eqref{SQP-MS} to update optimization variables:
    		\Statex $w_i^\pl \,\gets w_i^{\o} + \Delta w_i^{\o}$ and $\lambda_i^\pl \gets \lambda_i^{\o} + \Delta \lambda_i^{\o}$.  \Comment{full step} \vspace{1mm}
    		\Statex \texttt{Block-wise TR1 Jacobian updates}
    		\For{$i=0,\ldots,N-1$} \textbf{in parallel}
		\State Choose $\alpha_i^{\o} = \alpha_{\mathrm{F},i}^{\o}$ or $\alpha_i^{\o} = \alpha_{\mathrm{A},i}^{\o}$ via some decision rule.
    		\State $A^{\pl}_{i} \gets A^{\o}_{i} + \alpha_i^{\o} \left(y_i^{\o} - A^{\o}_{i}s_{i}^{\o}\right)\left(\gamma_i^{\o^\top} - \sigma_{i}^{\o^\top} A^{\o}_{i} \right)$.
    		\EndFor
    		\Ensure $w^{\pl}_{i} = (x^{\pl}_{i},u^{\pl}_{i})$, $\lambda_{i}^{\pl}$ and $A^{\pl}_{i}$ for $i = 0, \ldots, N-1$.
    	\end{algorithmic}
    \end{algorithm}

The complete adjoint-based SQP method that uses parallelizable block-wise TR1 Jacobian updates is summarized in Algorithm~\ref{alg:block_TR1}. Note that, for simplicity, the SQP algorithm is presented as a full-step method without any globalization or step-length selection strategies to ensure convergence to a local minimum~\cite{nocedal2006numerical}. This is also further motivated by the use of online algorithms for real-time nonlinear MPC as discussed in~\cite{Diehl2009c}.

\section{Convergence Results for Block-wise TR1-based SQP Method} \label{sec:CONV}

For the convergence analysis of sequential quadratic programming, it is standard to rely on a result that the active set, i.e., the set of active inequality constraints in the QP subproblems is stable in a neighbourhood around a local minimizer of the nonlinear program~\cite{nocedal2006numerical}. This allows us to study the local convergence properties of the block-TR1 based SQP method under the assumption that the active set has already been fixed, resulting, locally, in an equality constrained problem.

\subsection{Stability of the active set and local convergence}

Let us start by briefly repeating the result from~\cite{diehl2010adjoint} on the stability of the active set in the QP subproblems near the NLP solution and the corresponding conditions on local convergence properties for an adjoint-based SQP method with inexact Jacobians.
\begin{theorem} \label{AS-conv}
(Stability of active set and local convergence) Let the NLP solution vectors $w^*$, $\lambda^*$ be given and assume that:
\begin{enumerate}[label=(\roman*)]
\item at $w^*$ LICQ holds, and there exist Lagrange multiplier values $\mu^*$ such that $(w^*,\lambda^*,\mu^*)$ satisfies the KKT conditions in~\eqref{KKT-RK4}.
\item at $w^*$ strict complementarity holds, i.e., the multipliers $\mu_A^*$ of the active inequalities $P_A w^* = p_A$ satisfy $\mu_A^* > 0$, where $P_A$ is a matrix consisting of all rows of $P$ that correspond to the active inequalities at the NLP solution.
\item there are two sequences of uniformly bounded matrices $(A^k,H^k)$, each $H^k$ positive semidefinite on the null space of $A^k$, such that the sequence of matrices
\begin{equation}
J^k := \begin{bmatrix}
N^{\top}H^k & N^{\top}A^{k^\top}\\
A^{k} & \text{ }\\
P_A & \text{}
\end{bmatrix} \approx \frac{\partial \F}{\partial y}(y^k), \quad\text{where}\quad \F(y) := \begin{bmatrix} N^{\top} \nabla_w \LL(w,\lambda) \\ F(w) - X \\ P_A\, w - p_A \end{bmatrix}, \nonumber
\end{equation}
is uniformly bounded and invertible with a uniformly bounded inverse. Here, $N$ is a null space matrix with appropriate dimensions with orthonormal column vectors such that $N^\top N = \eye$ and $P_A\,N = 0$.
\item there is a sequence of iterates $y^k := (w^k,\lambda^k)$ generated according to
\begin{equation}
w^{k+1} = w^{k} + \Delta w^{k} \;\text{ and }\; \lambda^{k+1} = \lambda^k + \Delta \lambda^{k} = \lambda^{k+1}_{QP}, \nonumber
\end{equation}
where $\Delta w^k$ is the primal solution of the QP subproblem in~\eqref{SQP-MS} and $\lambda^{k+1}_{QP}$ denote the Lagrange multipliers corresponding to the equality constraints~\eqref{SQP-MS:Kstep}. Each iteration can be written in compact form as $y^{k+1} = y^{k} - J^{k^{-1}} \F(y^k)$.
\item there exists $\kappa < 1$ such that, for all $k \in \mathbb{N}$, it can be guaranteed that
\begin{equation}
\left\Vert J^{{k+1}^{-1}} \left( J^{k} - \frac{\partial \F}{\partial y}(y^k + t\Delta y^k)\right)\Delta y^k \right\Vert \leq \kappa \Vert  \Delta y^k \Vert , \quad \forall t \in [0, 1].
\end{equation}
\end{enumerate}
Then, there exists a neighbourhood $\bar{\mathcal{N}}$ of $(w^*,\lambda^*)$ such that for all initial guesses $(w^0,\lambda^0) \in\bar{\mathcal{N}}$ the sequence $(w^k, \lambda^k)$ converges q-linearly towards $(w^*, \lambda^*)$ with rate $\kappa$, and the solution of each QP~\eqref{SQP-MS} has the same active set as $w^*$.
\end{theorem}

In addition to the latter result that guarantees a q-linear local convergence rate in a neighbourhood of the NLP solution, the following theorem states a condition under which q-superlinear local convergence can be obtained instead.

\begin{theorem} \label{superlinear}
	(Superlinear convergence)
	If the equality
	\begin{equation}
	\underset{k \rightarrow \infty}{\text{lim}} \begin{bmatrix}
	N^{\top}H^k N & N^{\top}A^{k^\top}\\
	A^{k} N & \zero
	\end{bmatrix} = \begin{bmatrix}
	N^{\top}\nabla_{ww}^{2}\mathcal{L}(w^\star, \lambda^\star) N & N^{\top}\frac{\partial F}{\partial w}(w^\star)^\top\\
	\frac{\partial F}{\partial w}(w^\star) N & \zero
	\end{bmatrix},
	\end{equation}
	holds in addition to the assumptions of Theorem~\ref{AS-conv}, then the local convergence rate is q-superlinear instead.
\end{theorem}
The proofs for both Theorem~\ref{AS-conv} and~\ref{superlinear} can be found in~\cite{diehl2010adjoint} for an adjoint-based SQP method with inexact Jacobians that matches our problem formulation.

\subsection{Convergence of the block-wise TR1 Jacobian updates}

Theorem~\ref{AS-conv} holds for a general class of constraint Jacobian and Hessian approximation matrices $(A^k,H^k)$. Therefore, we have to show that our block-wise TR1 updates produce a sequence of block-structured matrices that converge to the exact Jacobian, which is itself block-structured, projected onto the null space of the active inequality constraint matrix $P_A$. Namely, defining a null space matrix $N$ as in Theorem~\ref{AS-conv}, we need to prove that the following holds
\begin{equation}
\lim_{k\rightarrow \infty} \left\Vert \left(A^{k}_{i} - \frac{\partial F_i}{\partial w}(w^*_{i}) \right) N_i \right\Vert = 0, \quad \forall i = 0,\ldots,N-1,
\end{equation}
where $N_i$ is the projection of the null space matrix $N$ in the variable space corresponding to block $i$. 
The only non-zero entries that are inexact in the Jacobian approximation matrix $A^k$ are those corresponding to the block-TR1 matrices $A^{k}_{i}$, $i = 0,\ldots,N-1$. 

\begin{assumption} Let us make the following assumptions:
\begin{itemize}
	\item[(AS1)] The Lagrangian function is twice continuously differentiable.
	\item[(AS2)] The function $\nabla_{w} F(w)$ is Lipschitz continuous, i.e., there exists a constant $c_3$ such that $
	\Vert \nabla_{w} F(w_1) - \nabla_{w} F(w_2) \Vert \le c_3 \Vert w_1 - w_2 \Vert
	$, for any $w_1$, $w_2$.
	\item[(AS3)] Let $\{(w^k,\lambda^k)\}$ be a sequence of iterates generated by our block-TR1 based SQP method in Algorithm~\ref{alg:block_TR1}, with a corresponding sequence of update parameters $\{\alpha^{k}_{i}\}$, while satisfying the skipping criterion in eqs.~\eqref{skip1}-\eqref{skip2}. 
	\item[(AS4)] The SQP iterates $\{(w^{k},\lambda^{k})\}$ converge to a limit point $(w^*,\lambda^*)$.
	\item[(AS5)] There is $k_0$ such that the active set is stable for all iterates $k \ge k_0$.
	\item[(AS6)] For each block $i$, the sequence of projections of $\{s^{k}\}$ on the subspace associated with block $i$, namely $\{s^{k}_{i}\}$ is uniformly linearly independent in the projected null space $N_i$, i.e., there exist $c_4>0$ and $l \geq q_i$ such that, for each $k_i \geq k_0$, there exist $q_i$ distinct indices $k^{j}_{i}$ with $k_i \leq k^{1}_{i} < \ldots < k^{q_i}_{i} \leq k_i + l$, $s^{k^{j}_{i}}_{N,i} \in \mathbb{R}^{q_i}$, $s^{k^{j}_{i}}_{i} = N_{i} s^{k^{j}_{i}}_{N,i}$, $j=1,\ldots,q_i$ and the minimum singular value $\sigma_{min}(S^{k_i}_{N_i})$ of the matrix
\begin{equation}
S^{k_{i}}_{N_i} = \begin{bmatrix}
\frac{s^{k^{1}_{i}}_{N,i}}{\Vert s^{k^{1}_{i}}_{N,i} \Vert} & \ldots & \frac{s^{k^{q_i}_{i}}_{N,i}}{\Vert s^{k^{q_i}_{i}}_{N,i} \Vert}
\end{bmatrix}
\end{equation}
is bounded below by $c_4$, i.e., $\sigma_{min}(S^{k_i}_{N_i}) \geq c_4$.
\end{itemize}
\label{as:conv}
\end{assumption}

Note that the assumptions $(AS1)$-$(AS5)$ are relatively mild and quite standard in Newton-type convergence analysis of SQP methods~\cite{diehl2010adjoint}. Especially, condition~$(AS5)$ holds due to the local stability result in Theorem~\ref{AS-conv} for the active set near the NLP solution. Even though~$(AS6)$ seems relatively strong, a very similar assumption is made in existing convergence results for quasi-Newton type matrix update schemes~\cite{conn1991convergence,diehl2010adjoint}. Here, we only require uniform linear independence inside each block $i$. 
We proceed now to prove the convergence of the quasi-Newton block-structured constraint Jacobian approximation matrices, using ideas from~\cite{conn1991convergence} and~\cite{diehl2010adjoint}. We start by first showing an intermediate result in the following lemma.

\begin{lemma}
	Given $(AS1)$-$(AS3)$ in Assumption~\ref{as:conv}, then the following holds for each Jacobian block matrix approximation
	\begin{subequations}
	\begin{alignat}{5}
	\left\Vert  y^{k}_{i} - A^{l}_{i}s^{k}_{i} \right\Vert &\leq \frac{c_3}{c_1}\left(\frac{2}{c_1^2} + 1 \right)^{l-k} \eta^{l,k}_{i} \Vert s^{k}_{i}\Vert, \quad &&\forall l \geq k +1, \label{eq:bound_f} \\
	\left\Vert \gamma^{k}_{i} - A^{l^\top}_{i} \sigma^{k}_{i}\right\Vert &\leq \frac{c_3}{c_1}\left(\frac{2}{c_1^2} + 1 \right)^{l-k} \eta^{l,k}_{i} \Vert \sigma^{k}_{i}\Vert, \quad &&\forall l \geq k +1, \label{eq:bound_a}
	\end{alignat}
	\end{subequations}
	where $i = 0, \ldots, N-1$ and $\eta^{l,k}_{i} = \max\{ \Vert  w^{p}_{i} - w^{s}_{i}\Vert  \text{  }  | \text{  } k \leq s \leq p \leq l\}$ is defined.
	
	\label{lem:first}
\end{lemma}

\begin{proof}
	Our proof follows closely the proof of Lemma~4.1 in~\cite{diehl2010adjoint} but extended to our block-structured method and generalized to include both the forward and adjoint TR1 Jacobian update formulas.\\
	\\
	\emph{Step~1: Eq.~\eqref{eq:bound_f} based on forward Jacobian update}\\
	We start by showing the result of Eq.~\eqref{eq:bound_f} when $\alpha_i = \alpha_{\mathrm{F},i}$. The proof is by induction for each block $i = 0, \ldots, N-1$. For $l = k+1$, we know that $y^{k}_{i} - A^{l}_{i}s^{k}_{i} = 0$ based on the forward update. Assume that the result in~\eqref{eq:bound_f} holds for all $j = k+1, \ldots, l$. Then, we have the following
	\begin{equation}
	\begin{aligned}
	\Vert y^{k}_{i} - A^{l+1}_{i}s^{k}_{i} \Vert &= \left \Vert y^{k}_{i} - A^{l}_{i} s^{k}_{i} - \alpha_{\mathrm{F},i}^l\, \rho^{l}_{i} \tau^{l^\top}_{i} s^{k}_{i} \right\Vert \\
	&\leq \Vert  y^{k}_{i} - A^{l}_{i}s^{k}_{i} \Vert + \left| \frac{(\tau^{l}_{i},s^{k}_{i})}{(\tau^{l}_{i},s^{l}_{i})} \right| \text{ } \Vert \rho^{l}_{i} \Vert,
	\end{aligned}
	\end{equation}
	where $\tau^{l}_{i} = \gamma^{l}_{i} - A^{l^\top}_{i} \sigma^{l}_{i}$ and $\rho^{l}_{i} = y^{l}_{i} - A^{l}_{i}s^{l}_{i}$ such that $\alpha_{\mathrm{F},i}^l = \frac{1}{(\tau^{l}_{i},s^{l}_{i})}$. Then, using the result in Eq.~\eqref{eq:bound_f}, we can write
	\begin{equation}
	\begin{aligned}
	\left| (\tau^{l}_{i},s^{k}_{i}) \right| &= \left| (\gamma^{l}_{i} - A_{i}^{l^\top}\sigma^{l}_{i},s^{k}_{i}) \right| \leq  \left| (\gamma^{l}_{i} , s^{k}_{i}) - (\sigma^{l}_{i}, y^{k}_{i}) \right| + \left| (\sigma^{l}_{i},y^{k}_{i}) -  (\sigma^{l}_{i} , A^{l}_{i}s^{k}_{i}) \right| \\
	& \leq \left| (\gamma^{l}_{i} , s^{k}_{i}) - (\sigma^{l}_{i}, y^{k}_{i}) \right| +  \frac{c_3}{c_1}\left(\frac{2}{c_1^2} + 1 \right)^{l-k} \eta^{l,k}_{i} \left\Vert \sigma^{l}_{i} \right\Vert \left\Vert s^{k}_{i} \right\Vert.
	\end{aligned} \label{eq:intermRes1}
	\end{equation}
	Using the mean-value theorem, it follows that
	\begin{equation}
	\begin{aligned}
	\left|(\gamma^{l}_{i} , s^{k}_{i}) - (\sigma^{l}_{i}, y^{k}_{i}) \right| &= \left| \sigma^{l^\top}_{i} \left( \frac{\partial F_i}{\partial w}(w^{l}_{i} + s^{l}_{i}) - \int_{0}^{1}{\frac{\partial F_i}{\partial w}(w^{k}_{i} + t\,s^{k}_{i}) \dd t} \right) s^{k}_{i} \right| \\
	& \leq c_3\, \eta^{l+1,k}_{i} \left\Vert \sigma^{l}_{i} \right\Vert  \left\Vert s^{k}_{i} \right\Vert,
	\end{aligned} \label{eq:meanValueRes}
	\end{equation}
	based on the Lipschitz continuity in $(AS2)$.
	From the skipping condition in~\eqref{skip1}, we know that $\left| (\tau^{l}_{i},s^{l}_{i}) \right| \geq c_1 \left\Vert \sigma_i^{l} \right\Vert \, \left\Vert \rho^{l}_{i} \right\Vert$. In addition, given that $\eta^{l,k}_{i} \leq \eta^{l+1,k}_{i}$, we obtain
	\begin{equation}
	\begin{aligned}
	\left\Vert y^{k}_{i} - A^{l+1}_{i}s^{k}_{i} \right\Vert &\leq \frac{c_3}{c_1}\left(\frac{2}{c_1^2} + 1 \right)^{l-k} \eta^{l,k}_{i} \left\Vert s^{k}_{i} \right\Vert  \\
	&+ \left( c_3 \eta^{l+1,k}_{i} +  \frac{c_3}{c_1}\left(\frac{2}{c_1^2} + 1 \right)^{l-k} \eta^{l,k}_{i} \right) \frac{\left\Vert \sigma^{l}_{i}\right\Vert \left\Vert s^{k}_{i}\right\Vert }{\left|(\tau^{l}_{i},s^{l}_{i})\right|} \left\Vert \rho^{l}_{i}\right\Vert  \\
	& \leq \frac{c_3}{c_1}\left(\frac{2}{c_1^2} + 1 \right)^{l+1-k} \eta^{l+1,k}_{i} \left\Vert s^{k}_{i} \right\Vert. 
	\end{aligned}
	\end{equation}\\
	\\
	\emph{Step~2: Eq.~\eqref{eq:bound_f} based on adjoint Jacobian update}\\
	Let us continue this proof by induction for Eq.~\eqref{eq:bound_f}, based on the adjoint Jacobian update formula. First, we derive the following error bound for the adjoint Jacobian update in case $l = k+1$ in Eq.~\eqref{eq:bound_f}
		\begin{equation}
		\begin{aligned}
		\left\Vert y^{k}_{i} - A^{k+1}_{i}s^{k}_{i} \right\Vert &= \left\Vert y^{k}_{i} - A^{k}_{i} s^{k}_{i} - \frac{1}{(\sigma_i^k, \rho^{k}_{i})}\, \rho^{k}_{i} \tau^{k^\top}_{i} s^{k}_{i} \right\Vert = \left\Vert \rho^{k}_{i} - \frac{1}{(\sigma_i^k, \rho^{k}_{i})}\, \rho^{k}_{i} \tau^{k^\top}_{i} s^{k}_{i} \right\Vert \\
		&= \left| 1 - \frac{(\tau^{k}_{i}, s^{k}_{i})}{(\sigma_i^k, \rho^{k}_{i})} \right| \left\Vert \rho^{k}_{i} \right\Vert = \left| \frac{(\sigma_i^k, y^{k}_{i} - A^{k}_{i} s^{k}_{i}) - (\gamma^{k}_{i} - A^{k^\top}_{i} \sigma^{k}_{i}, s^{k}_{i})}{(\sigma_i^k, \rho^{k}_{i})} \right| \left\Vert \rho^{k}_{i} \right\Vert \\
		&= \frac{\left|(\sigma_i^k, y^{k}_{i}) - (\gamma^{k}_{i}, s^{k}_{i})\right|}{\left|(\sigma_i^k, \rho^{k}_{i})\right|} \left\Vert \rho^{k}_{i} \right\Vert.
		\end{aligned}
		\end{equation}
		From the skipping conditions in~\eqref{skip1}-\eqref{skip2}, we know that $\left| (\sigma^{k}_{i},\rho^{k}_{i}) \right| \geq c_1 \left\Vert s^{k}_{i} \right\Vert \, \left\Vert \tau^{k}_{i} \right\Vert$ and $\left\Vert \rho^{k}_{i} \right\Vert \le \frac{\left| (\tau^{k}_{i}, s^{k}_{i}) \right|}{c_1 \left\Vert \sigma_i^{k} \right\Vert} \le \frac{\left\Vert \tau^{k}_{i}\right\Vert \left\Vert s^{k}_{i} \right\Vert}{c_1 \left\Vert \sigma_i^{k} \right\Vert}$ holds. We can use these lower and upper bounds to rewrite the latter expression as
		\begin{equation}
		\begin{aligned}
		\left\Vert y^{k}_{i} - A^{k+1}_{i}s^{k}_{i} \right\Vert &= \frac{\left|(\sigma_i^k, y^{k}_{i}) - (\gamma^{k}_{i}, s^{k}_{i})\right|}{\left|(\sigma_i^k, \rho^{k}_{i})\right|} \left\Vert \rho^{k}_{i} \right\Vert \le \frac{\left|(\sigma_i^k, y^{k}_{i}) - (\gamma^{k}_{i}, s^{k}_{i})\right|}{c_1 \left\Vert s^{k}_{i} \right\Vert \, \left\Vert \tau^{k}_{i} \right\Vert} \left\Vert \rho^{k}_{i} \right\Vert \\
		&\le \frac{\left|(\sigma_i^k, y^{k}_{i}) - (\gamma^{k}_{i}, s^{k}_{i})\right|}{c_1^2 \left\Vert \sigma_i^{k} \right\Vert} \\
		&\le \frac{c_3}{c_1^2}\,  \left\Vert s^{k}_{i} \right\Vert^2,
		\end{aligned}
		\end{equation}
		where we additionally used the result
		\begin{equation}
		\begin{aligned}
		\left|(\gamma^{k}_{i}, s^{k}_{i}) - (\sigma_i^k, y^{k}_{i})\right| &=  \left| \sigma^{k^\top}_{i}\left( \frac{\partial F_i}{\partial w}(w^{k+1}_{i}) - \int_{0}^{1}{\frac{\partial F_i}{\partial w}(w^{k}_{i} + t\,s^{k}_{i})\dd t} \right) s^{k}_{i} \right|  \\
		&\leq c_3\, \left\Vert \sigma^{k}_{i} \right\Vert  \left\Vert s^{k}_{i} \right\Vert^2.
		\end{aligned}
		\end{equation}
		Note that $\eta^{k+1,k}_{i} = \left\Vert s^{k}_{i} \right\Vert$ such that Eq.~\eqref{eq:bound_f} holds in case $l = k+1$.
	Assume that the result in~\eqref{eq:bound_f} holds for all $j = k+1, \ldots, l$. Then, we have the following
		\begin{equation}
		\begin{aligned}
		\Vert y^{k}_{i} - A^{l+1}_{i}s^{k}_{i} \Vert &= \left \Vert y^{k}_{i} - A^{l}_{i} s^{k}_{i} - \alpha_{\mathrm{A},i}^l\, \rho^{l}_{i} \tau^{l^\top}_{i} s^{k}_{i} \right\Vert \\
		&\leq \Vert  y^{k}_{i} - A^{l}_{i}s^{k}_{i} \Vert + \left| \frac{(\tau^{l}_{i},s^{k}_{i})}{(\sigma_i^l, \rho^{l}_{i})} \right| \text{ } \Vert \rho^{l}_{i} \Vert,
		\end{aligned}
		\end{equation}
		for the adjoint Jacobian update formula in which $\alpha_{\mathrm{A},i} = \frac{1}{(\sigma_i^l, \rho^{l}_{i})}$. 
		From the skipping conditions in~\eqref{skip1}-\eqref{skip2}, we know that $\left| (\sigma^{l}_{i},\rho^{l}_{i}) \right| \geq c_1 \left\Vert s^{l}_{i} \right\Vert \, \left\Vert \tau^{l}_{i} \right\Vert$ and $\left\Vert \rho^{l}_{i} \right\Vert \le \frac{\left| (\tau^{l}_{i}, s^{l}_{i}) \right|}{c_1 \left\Vert \sigma_i^{l} \right\Vert} \le \frac{\left\Vert \tau^{l}_{i}\right\Vert \left\Vert s^{l}_{i} \right\Vert}{c_1 \left\Vert \sigma_i^{l} \right\Vert}$ holds such that $\frac{\left\Vert \sigma^{l}_{i}\right\Vert \left\Vert \rho^{l}_{i}\right\Vert}{\left| (\sigma^{l}_{i},\rho^{l}_{i}) \right|} \le \frac{1}{c_1}\frac{\left\Vert \sigma^{l}_{i}\right\Vert \left\Vert \rho^{l}_{i}\right\Vert}{\left\Vert s^{l}_{i} \right\Vert \, \left\Vert \tau^{l}_{i} \right\Vert} \le \frac{1}{c_1^2}$. In addition, given eqs.~\eqref{eq:intermRes1} and~\eqref{eq:meanValueRes} and given that $\eta^{l,k}_{i} \leq \eta^{l+1,k}_{i}$, we obtain
		\begin{equation}
		\begin{aligned}
		\left\Vert y^{k}_{i} - A^{l+1}_{i}s^{k}_{i} \right\Vert &\leq \frac{c_3}{c_1}\left(\frac{2}{c_1^2} + 1 \right)^{l-k} \eta^{l,k}_{i} \left\Vert s^{k}_{i} \right\Vert  \\
		&+ \left( c_3 \eta^{l+1,k}_{i} +  \frac{c_3}{c_1}\left(\frac{2}{c_1^2} + 1 \right)^{l-k} \eta^{l,k}_{i} \right) \frac{\left\Vert \sigma^{l}_{i}\right\Vert \left\Vert s^{k}_{i}\right\Vert }{\left| (\sigma^{l}_{i},\rho^{l}_{i}) \right|} \left\Vert \rho^{l}_{i}\right\Vert  \\
		&\leq \frac{c_3}{c_1} \left(\frac{1}{c_1} + \left(\frac{1}{c_1^2} + 1 \right) \left(\frac{2}{c_1^2} + 1 \right)^{l-k}\right) \eta^{l+1,k}_{i} \left\Vert s^{k}_{i} \right\Vert  \\
		&\leq \frac{c_3}{c_1} \left(\frac{2}{c_1^2} + 1 \right)^{l+1-k} \eta^{l+1,k}_{i} \left\Vert s^{k}_{i} \right\Vert.
		\end{aligned}
		\end{equation}

	Note that the induction proof of step~1 and~2 implies that Eq.~\eqref{eq:bound_f} additionally holds when switching between the forward and adjoint Jacobian update formulas. A similar induction-based proof can be used to show the result of Eq.~\eqref{eq:bound_a} for the dynamic block-TR1 Jacobian updates.
\end{proof}

Now, we present the resulting theorem on the convergence of the Jacobian approximation for the block-wise TR1 scheme under any sequence of decision rules that select the adjoint or forward updates at every iteration $k$ and for each block $i = 0, \ldots, N-1$.
\begin{theorem}
Given $(AS1)$-$(AS6)$ in Assumption~\ref{as:conv}, then the following holds for each Jacobian block matrix $i = 0, \ldots, N-1$
\begin{equation}
\lim_{k \rightarrow \infty} \left\Vert \left(A^{k}_{i} - \frac{\partial F_i}{\partial w_i}(w^{*}_{i})\right) N_{i} \right\Vert = 0, \label{eq:jac_conv}
\end{equation}
such that the following holds for the complete Jacobian approximation
\begin{equation}
\lim_{k \rightarrow \infty} \left\Vert \left(A^{k} - \frac{\partial F}{\partial w}(w^{*})\right) N \right\Vert = 0. \label{eq:jac_conv2}
\end{equation}

\label{thm:convergence}
\end{theorem}

\begin{proof}
Based on the inequality $\Vert w^{p}_{i} - w^{s}_{i}\Vert \leq \Vert  w^{p}_{i} - w^{*}_{i}\Vert + \Vert  w^{s}_{i} - w^{*}_{i}\Vert$ and using the definition $\eta^{l,k}_{i} = \max\{ \Vert  w^{p}_{i} - w^{s}_{i}\Vert  \text{  }  | \text{  } k \leq s \leq p \leq l\}$, one obtains
\begin{equation}
\eta^{k+l+1,k}_{i} \leq 2\, \nu^{k}_{i} \;\text{ for }\; \nu^{k}_{i} = \max\{ \Vert w^{s}_{i} - w^{*}_{i}\Vert  \text{  }  | \text{  } k \leq s \leq k+l+1\},
\end{equation}
for $l \geq q_i$ and $q_i$ is defined as in Assumption~\ref{as:conv}.
In addition, the following holds
\begin{subequations}
\begin{alignat}{3}
\left\Vert  y^{j}_{i} - \frac{\partial F_i}{\partial w_i}(w^{*}_{i}) s^{j}_{i}\right\Vert &=  \left\Vert  \Bigg(\int_{0}^{1}{\frac{\partial F_i}{\partial w_i}(w^{j}_i + ts^{j}_i) \dd t }\Bigg) s^{j}_i - \frac{\partial F_i}{\partial w_i}(w^{*}_{i}) s^{j}_{i}\right\Vert \\
&= \left\Vert \Bigg(\int_{0}^{1}{\frac{\partial F_i}{\partial w_i}(w^{j}_i + ts^{j}_i) \dd t - \frac{\partial F_i}{\partial w_i}(w^{*}_{i}})\Bigg) s^{j}_i \right\Vert \\
&\leq c_3 \nu^{k}_{i} \left\Vert s^{j}_i\right\Vert, 
\end{alignat}
\end{subequations}
at an iteration $j$, where $k \leq j \leq k + l$, regardless of whether the forward or adjoint Jacobian update formula has been used. Moreover, from Lemma~\ref{lem:first}, we have that
\begin{equation}
\begin{aligned}
\left\Vert  y^{j}_{i} - A^{k+l+1}_{i}s^{j}_{i} \right\Vert &\leq \frac{c_3}{c_1}\left(\frac{2}{c_1^2} + 1 \right)^{k+l+1-j} \eta^{k+l+1,j}_{i} \Vert s^{j}_{i}\Vert, \quad k \leq j \leq k + l, \\
&\leq 2\, \frac{c_3}{c_1}\left(\frac{2}{c_1^2} + 1 \right)^{l+1} \nu^{k}_{i} \Vert s^{j}_{i}\Vert.
\end{aligned}
\end{equation}
We use the triangle inequality to obtain
\begin{subequations}
\begin{alignat}{2}
\left\Vert \left( A^{k+l+1}_{i} -  \frac{\partial F_i}{\partial w_i}(w^{*}_{i}) \right) \frac{s^{j}_{i}}{\left\Vert s^{j}_{i} \right\Vert} \right\Vert  &\leq \frac{1}{\left\Vert s^{j}_{i} \right\Vert} \left(\left\Vert  y^{j}_{i} - A^{k+l+1}_{i}s^{j}_{i} \right\Vert + \left\Vert  y^{j}_{i} - \frac{\partial F_i}{\partial w_i}(w^{*}_{i}) s^{j}_{i}\right\Vert \right)\\
&\leq \Bigg( 2\, \frac{c_3}{c_1}\left(\frac{2}{c_1^2} + 1 \right)^{l+1} + c_3 \Bigg) \nu^{k}_{i},
\end{alignat}
\end{subequations}
which holds for a sequence of indices $j = k_i^1, \ldots, k_i^{q_i}$. Then, we use the linear independence condition~$(AS6)$ in Assumption~\ref{as:conv} that guarantees both existence of the inverse $(S^{k_i}_{N_i})^{-1}$ and the upper bound $\Vert (S^{k_i}_{N_i})^{-1}\Vert \leq 1 / c_4$, such that
\begin{subequations}
\begin{alignat}{2}
\left\Vert \left( A^{k+l+1}_{i} - \frac{\partial F_i}{\partial w_i}(w^{*}_{i}) \right) N_i \right\Vert &\leq \frac{1}{c_4} \left\Vert \left( A^{k+l+1}_{i} - \frac{\partial F_i}{\partial w_i}(w^{*}_{i}) \right) N_i S^{k_i}_{N_i} \right\Vert \\
&\leq c_5\, \nu^{k}_{i},
\end{alignat}
\end{subequations}
where $c_5 = \frac{c_3}{c_4} \Bigg( \frac{2}{c_1}\left(\frac{2}{c_1^2} + 1 \right)^{l+1} + 1 \Bigg) \sqrt{q_i}$ has been defined. Lastly, the result in Eq.~\eqref{eq:jac_conv} follows from the fact that assumption~$(AS4)$ implies that $\nu^{k}_{i}$ tends to zero.
Note that this asymptotic result holds regardless of which Jacobian update (adjoint or forward TR1 formula) is performed for each block $i = 0, \ldots, N-1$. The same convergence result then holds for the complete Jacobian matrix in~\eqref{eq:jac_conv2}, based on separability of the active inequality constraints and of the nonlinear constraint functions.

\end{proof}



\subsection{Local rate of linear convergence for Gauss-Newton based SQP}

One iteration of the adjoint-based Gauss-Newton SQP method solves the linear system in Eq.~\eqref{QP-KKT}, which can be written in the following compact form
\begin{equation}
\Jin(z^k) \Delta z = -\F(z^k), \label{eq:compactNewton}
\end{equation}
where $\F(\cdot)$ denotes the KKT optimality conditions in the right-hand side of Eq.~\eqref{QP-KKT}.
Let us define regularity for a local minimizer $z^\star := (w^\star, \lambda^\star, \mu^\star)$ of the NLP, given a particular set of active inequality constraints. For this purpose, we rely on the linear independence constraint qualification~(LICQ) and the second order sufficient conditions~(SOSC) for optimality, of which the latter requires that the Hessian of the Lagrangian is strictly positive definite in the directions of the critical cone~\cite{nocedal2006numerical}.
\begin{definition}
	\label{as:regular}
	A minimizer of an equality constrained NLP is called a regular KKT point, if both LICQ and SOSC are satisfied at this KKT point.
\end{definition}
The convergence of this Newton-type optimization method then follows the classical and well-known local contraction theorem from~\cite{Bock1987,Deuflhard2011,diehl2010adjoint,Potschka2011}.
We use a particular version of this theorem from~\cite{Diehl2016,Quirynen2018}, providing sufficient and necessary conditions for the existence of a neighbourhood of the solution where the Newton-type iteration converges locally. Let $\rho(P)$ denote the spectral radius, i.e., the maximum absolute value of the eigenvalues for the square matrix $P$.

\begin{theorem}[Local Newton-type contraction~\cite{Diehl2016}] 
	We consider the twice continuously differentiable function $\F(z)$ from Eq.~\eqref{QP-KKT} and the regular KKT point $\F(z^\star)=0$ from Definition~\ref{as:regular}. We then apply the Newton-type iteration in Eq.~\eqref{eq:compactNewton}, where $\Jin(z) \approx J(z)$ is additionally assumed to be continuously differentiable and invertible in a neighbourhood of the solution. 
	If all eigenvalues of the iteration matrix have a modulus smaller than one, i.e., if the spectral radius satisfies
	\begin{equation}
	\kappa^\star := \rho\left(\Jin(z^\star)^{-1} J(z^\star) - \mathbb{1}\right) < 1, \label{eq:kappa_star}
	\end{equation}
	then this fixed point $z^\star$ is asymptotically stable.
	Additionally, the iterates $z^k$ converge linearly to the KKT point $z^\star$ with the asymptotic contraction rate $\kappa^\star$ when initialized sufficiently close. On the other hand, the fixed point $z^\star$ is unstable if $\kappa^\star > 1$.
	\label{local_contraction}
\end{theorem}
A proof for Theorem~\ref{local_contraction} can be found in~\cite{Diehl2016,Quirynen2017}, based on nonlinear systems theory. Using this result, let us define the linear contraction rate for a Gauss-Newton method with exact Jacobian information
\begin{equation}
\kappa_{\mathrm{GN}}^\star := \rho\left(\begin{bmatrix}
H & (\frac{\partial F}{\partial w}-E)^\top & P_A^\top \\
\frac{\partial F}{\partial w}-E & \text{ } & \text{ } \\
P_A & \text{} & \text{} 
\end{bmatrix}^{-1} \begin{bmatrix}
\nabla_{w}^2 \LL & (\frac{\partial F}{\partial w}-E)^\top & P_A^\top \\
\frac{\partial F}{\partial w}-E & \text{ } & \text{ } \\
P_A & \text{} & \text{} 
\end{bmatrix} - \mathbb{1}\right) < 1, \label{eq:exactGN_rate}
\end{equation}
at the local solution point $z^\star := (w^\star, \lambda^\star, \mu^\star)$ of the KKT conditions. In what follows, we show that the local contraction rate for the block-TR1 Gauss-Newton SQP method 
\begin{equation}
\kappa_{\mathrm{BTR1}}^\star := \rho\left(\begin{bmatrix}
H & (A-E)^\top & P_A^\top \\
A-E & \text{ } & \text{ } \\
P_A & \text{} & \text{} 
\end{bmatrix}^{-1} \begin{bmatrix}
\nabla_{w}^2 \LL & (\frac{\partial F}{\partial w}-E)^\top & P_A^\top \\
\frac{\partial F}{\partial w}-E & \text{ } & \text{ } \\
P_A & \text{} & \text{} 
\end{bmatrix} - \mathbb{1}\right) < 1, \label{eq:TR1GN_rate}
\end{equation}
coincides with the exact Jacobian based linear convergence rate in~\eqref{eq:exactGN_rate}. The following result states that the eigenspectrum of the iteration matrix $\Jin(z^\star)^{-1} J(z^\star) - \mathbb{1}$ at the solution point $z^\star := (w^\star, \lambda^\star, \mu^\star)$ coincides with the eigenspectrum of the iteration matrix $\Jgn(z^\star)^{-1} J(z^\star) - \mathbb{1}$, using the notation $\sigma(P)$ to denote the spectrum, i.e., the set of eigenvalues for a matrix $P$.

\begin{lemma}
	For a regular KKT point $z^\star := (w^\star, \lambda^\star, \mu^\star)$, the eigenvalues of the block-TR1 based iteration matrix $\Jin(z^\star)^{-1} J(z^\star) - \mathbb{1}$ satisfy
	\begin{equation}
	\sigma\left( \Jin(z^\star)^{-1} J(z^\star) - \mathbb{1} \right) = \sigma\left( \Jgn(z^\star)^{-1} J(z^\star) - \mathbb{1} \right).\label{eq:spectrum}
	\end{equation}
\label{lem:spectrum}
\end{lemma}
\begin{proof}
	Let us define the eigenvalues $\eig$ of the iteration matrix $\Jin(z^\star)^{-1} J(z^\star) - \mathbb{1}$ as the zeros of
	\begin{equation}
	\text{det}\left(\Jin(z^\star)^{-1} J(z^\star) - (\eig+1)\mathbb{1} \right) = 0,
	\end{equation}
	which, given that the Jacobian approximation $\Jin$ is invertible, this is equivalent to
	\begin{equation}
	\text{det}\left(J(z^\star) - (\eig+1)\Jin(z^\star) \right) = 0. \label{eq:det_interm}
	\end{equation}
	This block matrix then reads as
	\begin{equation}
	J(z^\star) - (\eig+1)\Jin(z^\star) = 
	\begin{bmatrix}
	\nabla_{w}^2 \LL - (\eig+1) H & \left( \frac{\partial F}{\partial w} - (\eig+1) A\right)^\top + \eig E^\top  & -\eig P_A^\top \\
	\left( \frac{\partial F}{\partial w} - (\eig+1) A\right) + \eig E & \text{ } & \text{ } \\
	-\eig P_A & \text{} & \text{} 
	\end{bmatrix}.
	\end{equation}
	The result follows from Theorem~\ref{thm:convergence} that claims the following asymptotic result for the block-TR1 based Jacobian approximation
	\begin{equation}
	\lim_{k \rightarrow \infty} \left(A^{k} - \frac{\partial F}{\partial w}(w^{*})\right) N = \left(A - \frac{\partial F}{\partial w}(w^{*})\right) N = 0,
	\end{equation}
	where $N$ is a null space matrix with appropriate dimensions and orthonormal column vectors such that $N^\top N = \eye$ and $P_A\,N = 0$. 
	We rewrite Eq.~\eqref{eq:det_interm} as follows
	\begin{equation}
	\begin{aligned}
	&\text{det}\left(J(z^\star) - (\eig+1)\Jin(z^\star) \right) = \\
	&(-\eig)^{2\,n_A} \; \text{det}\left( \begin{bmatrix}
	\nabla_{w}^2 \LL - (\eig+1) H & \left( \frac{\partial F}{\partial w} - (\eig+1) A\right)^\top + \eig E^\top  & P_A^\top \\
	\left( \frac{\partial F}{\partial w} - (\eig+1) A\right) + \eig E & \text{ } & \text{ } \\
	P_A & \text{} & \text{} 
	\end{bmatrix} \right).
	\end{aligned}
	\end{equation}
	It can be verified that $\text{det}\left(J(z^\star) - (\eig+1)\Jin(z^\star) \right)=0$ holds for $\eig=0$ with an algebraic multiplicity of $2 \, n_A$ as well as for the values of $\eig$ that satisfy
	\begin{equation}
	\begin{aligned}
	&\text{det}\left(\begin{bmatrix} N^\top & \zero\\ \zero&\eye \end{bmatrix}
	\begin{bmatrix}
	\nabla_{w}^2 \LL - (\eig+1) H & \left( \frac{\partial F}{\partial w} - (\eig+1) A\right)^\top + \eig E^\top  \\
	\left( \frac{\partial F}{\partial w} - (\eig+1) A\right) + \eig E & \zero
	\end{bmatrix}
	\begin{bmatrix} N& \zero \\ \zero &\eye \end{bmatrix}\right) \\
	&= \text{det}\left(
	\begin{bmatrix}
	N^\top \Delta H\, N & N^\top\left( \frac{\partial F}{\partial w} - (\eig+1) A\right)^\top + \eig N^\top E^\top  \\
	\left( \frac{\partial F}{\partial w} - (\eig+1) A\right) N + \eig E \,N & \zero
	\end{bmatrix}\right) \\
	&=(-\eig)^{2\,n_F} \; \text{det}\left(
	\begin{bmatrix}
	N^\top \Delta H\, N & N^\top \left( \frac{\partial F}{\partial w} - E \right)^\top  \\
	\left( \frac{\partial F}{\partial w} - E \right) N & \zero
	\end{bmatrix}\right) = 0,
	\end{aligned}
	\end{equation}
	in the limit for $k \rightarrow \infty$, where the compact notation $\Delta H := \left(\nabla_{w}^2 \LL - (\eig+1) H \right)$ has been used for the Gauss-Newton Hessian approximation. Therefore, the eigenvalues of the iteration matrix $\Jin(z^\star)^{-1} J(z^\star) - \mathbb{1}$ for the proposed block-TR1 approach, evaluated at a regular KKT point, are equal to the eigenvalues of the iteration matrix $\Jgn(z^\star)^{-1} J(z^\star) - \mathbb{1}$ for the exact Jacobian based Gauss-Newton method. 
\end{proof}

\begin{corollary}
	Based on Lemma~\ref{lem:spectrum}, the linear contraction rate for the block-TR1 based optimization algorithm coincides with the linear contraction rate of the exact Jacobian based Gauss-Newton method $\kappa_{\mathrm{BTR1}}^\star = \kappa_{\mathrm{GN}}^\star$, when the iterates are sufficiently close to the regular KKT point $z^\star := (w^\star, \lambda^\star, \mu^\star)$. \label{cor:rate}
\end{corollary}

%

\subsection{Superlinear convergence for SQP with quasi-Newton Hessian updates}

Even though the majority of this article is focused on the generalized Gauss-Newton method for nonlinear least squares type optimization problems that occur frequently in predictive control applications, note that superlinear convergence results can be recovered when a block-structure preserving quasi-Newton method is additionally used to approximate the Hessian of the Lagrangian. For example, let us consider the following lemma that represents a block-structured or \emph{partitioned} version~\cite{Griewank1982,Griewank1982a} of the Broyden-Fletcher-Goldfarb-Shanno~(BFGS)~\cite{Asprion2014} or the symmetric rank-one~(SR1) formula~\cite{conn1991convergence,janka2016sr1} to approximate the block-diagonal Hessian matrix.

	\begin{theorem}
		Given $(AS1)$-$(AS6)$ in Assumption~\ref{as:conv}, then the following holds for each Hessian block matrix approximation
		\begin{equation}
		\lim_{k \rightarrow \infty} \left\Vert \left(H^{k}_{i} - \nabla_{w_i w_i}^{2}\mathcal{L}(w_i^\star, \lambda_i^\star)\right) N_{i} \right\Vert = 0,
		\end{equation}
		$i = 0, \ldots, N-1$, such that the following holds for the complete Hessian approximation
		\begin{equation}
		\lim_{k \rightarrow \infty} \left\Vert \left(H^{k} - \nabla_{w w}^{2}\mathcal{L}(w^\star, \lambda^\star)\right) N \right\Vert = 0.
		\end{equation}
		\label{thm:Hess_convergence}
\end{theorem}

Theorem~\ref{thm:Hess_convergence} on the convergence of a separable quasi-Newton type Hessian approximation method in combination with our main result in Theorem~\ref{thm:convergence} on the block-structured quasi-Newton type Jacobian update formula can be used directly to prove the following result on convergence of the reduced KKT matrix.
\begin{theorem}
	Given $(AS1)$-$(AS6)$ in Assumption~\ref{as:conv}, the following holds
\begin{equation}
\underset{k \rightarrow \infty}{\text{lim}} \left \Vert \begin{bmatrix}
N^{\top}H^k N & N^{\top}A^{k^\top}\\
A^{k} N & \zero
\end{bmatrix} - \begin{bmatrix}
N^{\top}\nabla_{ww}^{2}\mathcal{L}(w^\star, \lambda^\star) N & N^{\top}\frac{\partial F}{\partial w}(w^\star)^\top\\
\frac{\partial F}{\partial w}(w^\star) N & \zero
\end{bmatrix} \right \Vert = 0.
\end{equation}
	\label{thm:KKT_convergence}
\end{theorem}
Based on Theorem~\ref{superlinear}, the above result ensures q-superlinear convergence of the SQP iterates when using a quasi-Newton method to update both the block-structured Hessian and Jacobian matrices.
The proof for Theorem~\ref{thm:KKT_convergence}, based on the intermediate convergence results in Theorem~\ref{thm:convergence} and~\ref{thm:Hess_convergence} can be found in~\cite{diehl2010adjoint}. 


\section{Lifted Collocation Algorithm with Block-TR1 Jacobian Updates}\label{sec:TRD}

As mentioned earlier, implicit integration schemes are often used in direct optimal control because of their relatively high order of accuracy and their improved numerical stability properties~\cite{Hairer1991}. More specifically, problem formulations based on a system of stiff and/or implicit differential or differential-algebraic equations require the use of an implicit integration scheme. Collocation methods are a popular family of implicit Runge-Kutta methods. This section presents a novel lifted collocation algorithm based on tailored block-TR1 Jacobian updates. The standard lifted collocation method with exact Jacobian information was proposed in~\cite{quirynen2017lifted} as a structure-exploiting implementation of direct collocation, even though it shows similarities to multiple shooting.

\subsection{Direct collocation for nonlinear optimal control}
\label{sec:drco}

In direct transcription methods, such as direct collocation~\cite{Betts2010,Biegler1984}, the integration scheme and its intermediate variables are directly made part of the nonlinear optimization problem. In this context, where the simulation routine is defined implicitly as part of the equality constraints in the dynamic optimization problem, one typically relies on implicit integration schemes for their relatively high order of accuracy and improved numerical stability properties. The discrete-time optimal control problem can generally be written as
\begin{subequations} \label{DC-OCP}
	\begin{alignat}{5}
	\underset{X,\,U,\,K}{\text{min}} \quad &\sum_{i=0}^{N-1} l_i(x_i,u_i) + l_N(x_N)&& \label{DC:obj}\\
	\text{s.t.} \quad\;\; & \hat{x}_0 \;= \;x_0,  	\label{DC:initial}\\
	& x_{i} + B_{i} \, K_{i}  \;=\; x_{i+1}  , \quad &&i = 0, \ldots, N-1, \label{DC:Kstep}\\
	& G_i(x_{i},u_{i},K_{i}) \;= \;0,	\quad &&i = 0, \ldots, N-1,	\label{DC:Kvals}\\
	& P_i\, w_i \leq p_i, \quad &&i = 0, \ldots, N, \label{DC:setX}
	\end{alignat}
\end{subequations}
where the additional trajectory $K=[K_0^\top,\ldots, K_{N-1}^\top]^\top$ denotes the intermediate variables of the numerical integration method. These variables are defined implicitly by the equations in~\eqref{DC:Kvals}, such that the continuity condition reads as in Eq.~\eqref{DC:Kstep}. More specifically, the Jacobian $\frac{\partial G_i}{\partial K_{i}}(\cdot)$ will generally be invertible for an integration scheme applied to a well-defined set of differential equations in~\eqref{C:path}. A popular approach of this type is better known as direct collocation~\cite{biegler2010nonlinear}. It relies on a collocation method, a subclass of implicit Runge-Kutta~(IRK) methods~\cite{Hairer1991}, to accurately discretize the continuous time dynamics. In this case, the equations in~\eqref{DC:Kvals} define the collocation polynomial on each control interval $i = 0, \ldots, N-1$.

In a similar fashion as in Section~\ref{sec:DOC}, the adjoint-based SQP method can be applied directly to the direct collocation problem in~\eqref{DC-OCP} by solving the following convex QP subproblem at each iteration
\begin{subequations} \label{SQP-DC}
	\begin{alignat}{4}
	\underset{\Delta W,\,\Delta K}{\text{min}} \; &\sum_{i=0}^{N} \frac{1}{2} \Delta w_{i}^\top H_{i}^k\, \Delta w_{i} \,+\, h_{i}^{c^\top} \bvec \Delta w_{i} \\ \Delta K_{i} \evec &&\label{SQP-DC:obj} \\
	\text{s.t.} \quad\; & \Delta x_0 \;= \; \hat{x}_0 - x_0^k,  	\label{SQP-DC:initial}\\
	& e_{i}^k + \Delta x_{i} + B_i \, \Delta K_{i}  \,= \,  \Delta x_{i+1}, &&i = 0, \ldots, N-1, 	\label{SQP-DC:Kstep}  \\
	& c_{i}^k + D_{i}^k \Delta w_{i} + C_{i}^k \Delta K_{i} \;= \;0, &&i = 0, \ldots, N-1, \label{SQP-DC:Kvals}  \\
	& P_i\, \Delta w_i \leq p_{i}^k, \qquad \quad\qquad &&i = 0,\ldots, N, \label{SQP-DC:path}
	\end{alignat}
\end{subequations}
based on $c_i^k := G_i(w_i^k,K_i^k)$ and the Jacobian approximations $D_{i}^k \approx \frac{\partial G_i}{\partial w_i}(w_i^k,K_i^k)$ and $C_{i}^k \approx \frac{\partial G_i}{\partial K_i}(w_i^k,K_i^k)$. The corresponding gradient correction reads as
\begin{equation}
h_i^c := \bvec \nabla_{w_i} l(w_i^k) + \left(\frac{\partial G_i}{\partial w_i}(w_i^k,K_i^k) - D_{i}^k\right)^\top \omega_i^k \\ \left(\frac{\partial G_i}{\partial K_i}(w_i^k,K_i^k) - C_{i}^k\right)^\top \omega_i^k \evec, \label{eq:grad-DC}
\end{equation}
where $\omega_i^k$ denotes the current value of the multipliers for the nonlinear constraints in~\eqref{DC:Kvals} and $\lambda_i^k$ again denotes the multipliers for the continuity constraints in~\eqref{DC:Kstep}. 

\subsection{Tailored structure exploitation for direct collocation}

As mentioned earlier, the Jacobian matrix $\frac{\partial G_i}{\partial K_{i}}$ for the collocation equations needs to be invertible. Therefore, given an invertible approximation $C_{i}^\o \approx \frac{\partial G_i}{\partial K_i}(w_i^\o,K_i^\o)$, we can rewrite the linearized expression in Eq~\eqref{SQP-DC:Kvals} as follows
\begin{equation}
\Delta K_{i} = -C^{\o^{-1}}_{i} \left( c_{i}^\o + D_{i}^\o \Delta w_{i} \right). \label{eq:K_upd}
\end{equation}
By substituting the above expression for $\Delta K_{i}$ back into the direct collocation structured QP in~\eqref{SQP-DC}, one obtains the condensed but equivalent formulation
\begin{subequations} \label{SQP-LC}
	\begin{alignat}{4}
	\underset{\Delta W}{\text{min}} \quad &\sum_{i=0}^{N} \frac{1}{2} \Delta w_{i}^\top H_{i}^k\, \Delta w_{i} \,+\, \tilde{h}_{i}^{c^\top} \Delta w_{i} &&\label{SQP-LC:obj} \\
	\text{s.t.} \quad & \Delta x_0 \;= \; \hat{x}_0 - x_0^k,  	\label{SQP-LC:initial}\\
	& d_{i}^k + \Delta x_{i}- B_i \, C^{k^{-1}}_{i}D_{i}^k \Delta w_{i} \;= \;\Delta x_{i+1}, \quad &&i = 0, \ldots, N-1, \label{SQP-LC:Kstep}  \\
	& P_i\, \Delta w_i \leq p_{i}^k, \qquad \quad\;\qquad &&i = 0,\ldots, N, \label{SQP-LC:path}
	\end{alignat}
\end{subequations}
where $d_{i}^\o = e_{i}^\o - B_i\,C^{\o^{-1}}_{i} c_{i}^\o$ is defined and the condensed gradient reads as
\begin{equation}
\begin{aligned}
\tilde{h}^{c}_{i} &= \nabla_{w_i} l(w_i^\o) + \left(\frac{\partial G_i}{\partial w_i} - \frac{\partial G_i}{\partial K_i}\, C^{\o^{-1}}_{i} D_{i}^\o \right)^\top \omega_i^{\o},
 \end{aligned} \label{eq:cond_grad}
\end{equation}
given the original gradient correction in~\eqref{eq:grad-DC}.

Note that the resulting QP formulation in Eq.~\eqref{SQP-LC} is of the same problem dimensions and exhibits the same sparsity as the multiple shooting structured QP subproblem in Eq.~\eqref{SQP-MS}. Therefore, state of the art block-structured QP solvers can be used, for which an overview can be found in~\cite{Ferreau2017}.
After solving the condensed QP in~\eqref{SQP-LC}, the collocation variables can be obtained from the expansion step in Eq.~\eqref{eq:K_upd}. Based on the optimality conditions of the original direct collocation structured QP in~\eqref{SQP-DC}, the corresponding Lagrange multipliers can be updated as follows
\begin{equation}
\omega^{\pl}_{i} = \omega^{\o}_{i} - C^{\o^{-\top}}_{i} \left( \frac{\partial G_i}{\partial K_i}^\top \omega^{\o}_{i} + B_i^\top \lambda^{\pl}_{i} \right), \label{eq:lag_upd}
\end{equation}
where $\lambda^{\pl}_{i}$ denote the new values of the Lagrange multipliers for the continuity conditions in~\eqref{SQP-LC:Kstep} or in~\eqref{SQP-DC:Kstep}.

\subsection{Block-TR1 Jacobian update for lifted collocation}

The block-TR1 update formula from Eq.~\eqref{eq:BTR1} can be readily applied to the direct collocation equations, resulting in
\begin{equation}
 [D^{\pl}_{i} \ C^{\pl}_{i}] = [D^{\o}_{i} \ C^{\o}_{i}] + \alpha_i^{\o} \left(y_i^{\o} - [D^{\o}_{i} \ C^{\o}_{i}]\, s_{i}^{\o} \right) \left(\gamma_{i}^{\o^\top} - \sigma_{i}^{\o^\top}[D^{\o}_{i} \ C^{\o}_{i}] \right),
 \label{eq:LC-BTR1}
\end{equation}
where the quantities $\gamma_{i}^{\o^\top} = \sigma_{i}^{\o^\top} \frac{\partial G_i}{\partial (w_i, K_i)}(w^{\pl}_i,K^{\pl}_i)$ and $\sigma_{i}^\o = \omega^{\pl}_{i} - \omega^{\o}_{i}$ are defined. In addition, $s_i^\o := \bvec w_i^{\pl} - w_i^{\o} \\ K_i^{\pl} - K_i^{\o} \evec$ and $y_i^\o = G_i(w^{\pl}_{i},K^{\pl}_{i}) - G_i(w^{\o}_{i},K^{\o}_{i})$ is defined. In order to use this block-TR1 update formula in combination with the lifted collocation method, one needs to be able to efficiently form the condensed QP in Eq.~\eqref{SQP-LC}. For this purpose, we need to avoid the costly computations of the inverse matrix $C^{\o^{-1}}_{i}$ as well as the matrix-matrix multiplication $C^{\o^{-1}}_{i} D_{i}^\o$. In what follows, we present a procedure to directly obtain a rank-one update formula for the inverse matrix $C^{\pl^{-1}}_{i}$ and for the corresponding product $E^{\pl}_{i} := C^{\pl^{-1}}_{i} D^{\pl}_{i}$.


Based on the Sherman-Morrison formula, one can directly update the matrix inverse given the previous invertible approximation $C^{\o^{-1}}_{i} \approx \frac{\partial G_i}{\partial K_i}^{-1}$. Let us first rewrite the block-TR1 update from Eq.~\eqref{eq:LC-BTR1} as follows
	\begin{equation}
		D^{\pl}_{i} = D^{\o}_{i} + \alpha_i^{\o}\, \rho_i^{\o}\, \tau_{\mathrm{D},i}^{\o^\top} \quad \text{and} \quad
		C^{\pl}_{i} = C^{\o}_{i} + \alpha_i^{\o}\, \rho_i^{\o}\, \tau_{\mathrm{C},i}^{\o^\top},
	\end{equation}
where $\rho_i^\o = y_i^\o - [D^{\o}_{i} \ C^{\o}_{i}] s_{i}^\o$ and $[\tau_{\mathrm{D},i}^{\o^\top} \ \tau_{\mathrm{C},i}^{\o^\top}] = \gamma_{i}^{\o^\top} - \sigma_{i}^{\o^\top}[D^{\o}_{i} \ C^{\o}_{i}]$. The Sherman-Morrison formula then reads as
\begin{equation}
C^{\pl^{-1}}_{i} = C^{\o^{-1}}_{i} - \alpha_i^\o \beta_i^\o\, C^{\o^{-1}}_{i} \rho_i^\o \tau_{\mathrm{C},i}^{\o^\top} C^{\o^{-1}}_{i}, \label{eq:C_upd}
\end{equation}
where $\beta_i^\o = \frac{1}{1 \,+\, \alpha_i^\o \tau_{\mathrm{C},i}^{\o^\top} C^{\o^{-1}}_{i} \rho_i^\o}$.
Let us define $\tilde{\rho}_i^\o = C^{\o^{-1}}_{i} \rho_i^\o$ such that we obtain the following update for the condensed Jacobian
\begin{equation}
\begin{aligned}
 E^{\pl}_{i} &= C^{\pl^{-1}}_{i} D^{\pl}_{i} = C^{\o^{-1}}_{i}\left(D^{\o}_{i} + \alpha_i^\o \rho_i^\o \tau_{\mathrm{D},i}^{\o^\top}\right) \\
& - \alpha_i^\o \beta_i^\o C^{\o^{-1}}_{i} \rho_i^\o \tau_{\mathrm{C},i}^{\o^\top}C^{\o^{-1}}_{i}\left(D^{\o}_{i} + \alpha_i^\o \rho_i^\o \tau_{\mathrm{D},i}^{\o^\top}\right)\\
&= E^{\o}_{i} + \alpha_i^\o \tilde{\rho}_i^\o \tau_{\mathrm{D},i}^{\o^\top} - \alpha_i^\o \beta_i^\o \tilde{\rho}_i^\o \tau_{\mathrm{C},i}^{\o^\top}(E^{\o}_{i} + \alpha_i^\o\tilde{\rho}_i^\o \tau_{\mathrm{D},i}^{\o^\top}) \\
&= E^{\o}_{i} + \alpha_i^\o \tilde{\rho}_i^\o \tilde{\tau}_i^{\o^\top},
\end{aligned} \label{eq:LC-BTR1-v2}
\end{equation}
where $\tilde{\tau}_i^{\o^\top} = \tau_{\mathrm{D},i}^{\o^\top} - \beta_i^\o \tau_{\mathrm{C},i}^{\o^\top}(E^{\o}_{i} + \alpha_i^\o \tilde{\rho}_i^\o \tau_{\mathrm{D},i}^{\o^\top})$ has been defined. It is readily seen that the update for $E_i^\o$ in Eq.~\eqref{eq:LC-BTR1-v2} is a rank-one update for the condensed Jacobian matrix. As proposed in~\cite{Hespanhol2019}, corresponding low-rank update formulas for the condensed Hessian can be obtained for the special case of a pseudospectral method based on a global collocation polynomial.

\subsection{Lifted collocation SQP method with block-TR1 Jacobian updates}

It is important to stress that the novel block-TR1 update formula for the condensed Jacobian matrix $E^{\pl}_{i} = C^{\pl^{-1}}_{i} D^{\pl}_{i}$ in Eq.~\eqref{eq:LC-BTR1-v2} provides an efficient manner to directly compute the rank-one update to the matrices in the condensed QP formulation of Eq.~\eqref{SQP-LC}, without the need for a matrix factorization, inversion and without any matrix-matrix multiplications. Instead, the proposed implementation merely requires matrix-vector multiplications and outer products, resulting in a quadratic instead of cubic computational complexity with respect to the number of optimization variables within each control interval. However, this comes at the cost of a slightly increased memory footprint, since additionally the matrices $C^{-1}_{i}$ and $E_i$ need to be stored from one iteration to the next. The implementation of the lifted block-TR1 based SQP method for direct collocation is presented in Algorithm~\ref{alg:lifted_block_TR1}.

\begin{algorithm}[h]
	\caption{One lifted collocation SQP iteration with block-wise TR1 updates.}
	\label{alg:lifted_block_TR1}
	\begin{algorithmic}[1]
    		\Require $w^{\o}_{i} = (x^{\o}_{i},u^{\o}_{i})$, $K^{\o}_{i}$, $\lambda_{i}^{\o}$, $\omega_{i}^{\o}$, $C^{\o}_{i}$, $D^{\o}_{i}$, $C^{\o^{-1}}_{i}$ and $E^{\o}_{i}$.
    		\Statex \texttt{Problem linearization and QP preparation}
    		\State Formulate the QP in~\eqref{SQP-LC} with Jacobian matrices $E^{\o}_{i}$, Gauss-Newton Hessian approximations $H_i^{\o}$ and vectors $d_{i}^\o$, $p_{i}^\o$ and $\tilde{h}^{c}_{i}$ in~\eqref{eq:cond_grad} for $i = 0, \ldots, N-1$. \vspace{1mm}	
    		\Statex \texttt{Computation of Newton-type step direction}
    		\State Solve the QP subproblem in Eq.~\eqref{SQP-LC} to update optimization variables:
    		\Statex $w_i^\pl \,\gets w_i^{\o} + \Delta w_i^{\o}$ and $\lambda_i^\pl \gets \lambda_i^{\o} + \Delta \lambda_i^{\o}$. 	\Comment{full step} \vspace{1mm}
    		\Statex \texttt{Block-wise TR1 Jacobian updates}
    		\For{$i=0,\ldots,N-1$} \textbf{in parallel}
		\State Choose $\alpha_i^{\o} = \alpha_{\mathrm{F},i}^{\o}$ or $\alpha_i^{\o} = \alpha_{\mathrm{A},i}^{\o}$ via some decision rule.
    		\State $K^{\pl}_{i} \;\gets K_i^{\o} - C^{\o^{-1}}_{i}c_{i}^{\o} - E^{\o}_{i} \Delta w_i^{\o}$,
			\State $\omega^{\pl}_{i} \;\,\gets\, \omega^{\o}_{i} - C^{\o^{-\top}}_{i} \left( \frac{\partial G_i}{\partial K_i}^\top \omega^{\o}_{i} + B_i^\top \lambda^{\pl}_{i} \right)$, 	
			\State $D^{\pl}_{i} \;\,\gets D^{\o}_{i} + \alpha_i^{\o}\, \rho_i^{\o}\, \tau_{\mathrm{D},i}^{\o^\top}$ and  $C^{\pl}_{i} \,\gets C^{\o}_{i} + \alpha_i^{\o}\, \rho_i^{\o}\, \tau_{\mathrm{C},i}^{\o^\top}$,
			\State $C^{\pl^{-1}}_{i} \hspace{-1mm}\gets C^{\o^{-1}}_{i} - \alpha_i^\o \beta_i^\o\, \tilde{\rho}_i^\o \tau_{\mathrm{C},i}^{\o^\top} C^{\o^{-1}}_{i}$,
    		\State $E^{\pl}_{i} \;\,\gets E^{\o}_{i} + \alpha_i^\o \tilde{\rho}_i^\o \tilde{\tau}_i^{\o^\top}$.
    		\EndFor
    		\Ensure $w^{\pl}_{i}$, $K^{\pl}_{i}$, $\lambda_{i}^{\pl}$, $\omega_{i}^{\pl}$, $C^{\pl}_{i}$, $D^{\pl}_{i}$, $C^{\pl^{-1}}_{i}$ and $E^{\pl}_{i}$.
    		
	\end{algorithmic}
\end{algorithm}


\subsection{Convergence results for block-TR1 based lifted collocation}

We observe that the TR1 Jacobian updates of the lifted collocation implementation are equivalent to the updates of the direct collocation method. More specifically, the Jacobian approximation matrices are the same at each SQP iteration, regardless of whether we perform the condensing and expansion procedure for the collocation variables in the proposed lifted implementation of Algorithm~\ref{alg:lifted_block_TR1}. Therefore, the convergence properties shown in the previous section also hold for both the standard and lifted collocation based block-TR1 SQP method. 

\begin{corollary}
If the assumptions of Theorem~\ref{AS-conv} and Assumption~\ref{as:conv} hold, then the lifted collocation SQP method with block-wise TR1 Jacobian updates in Algorithm~\ref{alg:lifted_block_TR1}, with a Gauss-Newton Hessian approximation, produces iterates $\{w^{k}, \lambda^{k}, \mu^{k}\}$ that converge q-linearly within a neighbourhood around the KKT point $(w^{*}, \lambda^{*}, \mu^{*})$ of the NLP.
\end{corollary}

\begin{proof}
	
It follows from the equivalence of the SQP iterations between the direct and lifted collocation formulation based on the numerical condensing and expansion of the collocation variables in Eq.~\eqref{eq:K_upd}. In particular, the direct collocation QP subproblem~\eqref{SQP-DC} is a special case of the QP formulation in~\eqref{SQP-MS}, with additional intermediate variables and corresponding equations. The block-TR1 Jacobian matrix convergence results of Theorem~\ref{thm:convergence} therefore hold for direct collocation as well as for the proposed lifted implementation in Algorithm~\ref{alg:lifted_block_TR1}. 

\end{proof}


\section{Numerical Case Studies of Nonlinear Model Predictive Control}
\label{sec:caseStudies}

In this section, we illustrate numerically how the proposed block-TR1 SQP method can be used in the context of nonlinear MPC using an algorithm implementation based on the real-time iterations~(RTI), as proposed originally in~\cite{Diehl2005b} with exact Jacobian information. The approach is based on one block-TR1 SQP iteration per control time step, and using a continuation-based warm starting of the state and control trajectories from one time step to the next~\cite{Hespanhol2018}. Each iteration consists of two steps:
\begin{enumerate}
	\item \emph{Preparation phase}: discretize and linearize the system dynamics, linearize the remaining constraint functions, and evaluate the quadratic objective approximation to build the optimal control structured QP subproblem.
	\item \emph{Feedback phase}: solve the QP to update the current values for all optimization variables and obtain the next control input to apply feedback to the system.
\end{enumerate}
The proposed block-wise TR1 based Jacobian updates in Algorithm~\ref{alg:block_TR1} and~\ref{alg:lifted_block_TR1} become part of the preparation step, in order to construct the linearized continuity equations. Therefore, the feedback step remains unchanged and the Jacobian updates do not affect the computational delay between obtaining the new state estimate and applying the next control input value to the system. 


We validate the closed-loop performance of these novel block-TR1 based RTI algorithms by presenting numerical simulation results for two NMPC case studies. Motivated by real embedded control applications, we present the computation times for the proposed NMPC algorithms using the ARM Cortex-A53 processor in the Raspberry Pi~3. The block-sparse QP solution in the feedback phase will be carried out by the primal active-set method, called \texttt{PRESAS}, that was recently presented in~\cite{PRESAS}.
 


\subsection{Nonlinear MPC for a chain of spring-connected masses}

In our first case study, the control task is to return a chain of $\nm$ masses connected with springs to its steady state, starting from a perturbed initial configuration, without hitting a wall that is placed close to the equilibrium state configuration. The mass at one end is fixed, while the control input $u(t) \in \R^{3}$ to the system is the direct force applied to the mass at the other end of the chain. The state of each free mass $x^j := [p^{j^\top},v^{j^\top}]^\top \in \R^{6}$ consists in its position $p^j := [p_x^j,p_y^j,p_z^j]^\top \in \R^{3}$ and velocity $v^j \in \R^{3}$ for $j=1,\ldots,\nm-1$, such that the dynamic system can be described by the concatenated state vector $x(t) \in \R^{6(\nm-1)}$.
Similar to the work in~\cite{quirynen2017lifted}, the nonlinear chain of masses can be used to validate the computational performance and scaling of an optimal control algorithm for a range of numbers of masses $\nm$, resulting in a range of different problem dimensions. 
The nonlinear system dynamics and the resulting optimal control problem formulation can be found in~\cite{Wirsching2006}.

\subsubsection{Local convergence: Gauss-Newton SQP with block-TR1 Jacobian updates}

We illustrate the impact of the proposed block-wise TR1 Jacobian updates on the local convergence rate of the resulting inexact adjoint-based SQP algorithm. Figure~\ref{figtime:one} shows a comparison of the convergence between different SQP variants for the solution of the nonlinear chain of masses OCP. In particular, the comparison includes the exact Jacobian-based SQP method, the standard dense TR1 update~\cite{griewank2002constrained}, and the good and bad Broyden update formulas~\cite{broyden2000discovery}. For the proposed block-TR1 based SQP implementation, the figure illustrates both the adjoint and forward variant by using, respectively, the scaling factor in~\eqref{eq:alpha1} and~\eqref{eq:alpha2}. The performance of the block-TR1 method is additionally illustrated for an implementation where $\alpha_i$ is chosen dynamically, depending on which of the two variants results in the largest denominator in order to avoid the need to skip a block-wise Jacobian update.

It is known that an exact Jacobian-based SQP method with Gauss-Newton type Hessian approximation results in locally linear convergence, for which the asymptotic contraction rate depends on the optimal residual value in the least squares type objective~\cite{nocedal2006numerical}. It can be observed in Figure~\ref{figtime:one} that all three variants of the proposed block-wise TR1 update formula result in the same asymptotic rate of convergence as for the exact Jacobian based algorithm, i.e., the rate of convergence appears to be the same close to the local solution of the NLP. Note that this confirms numerically the result of Corollary~\ref{cor:rate}.
In addition, the block-wise TR1 Jacobian updates result in a smaller total number of SQP iterations, compared to the standard dense Jacobian update formulas for the particular example in Figure~\ref{figtime:one}. In the latter case, the direct application of a standard rank-one update formula destroys the block sparsity in the QP subproblems and is therefore computationally unattractive. 

\begin{figure*}[t]
	\centering 
	\includegraphics[trim={1.5in 1.7in 0.in 0.0in},clip,scale=0.33]{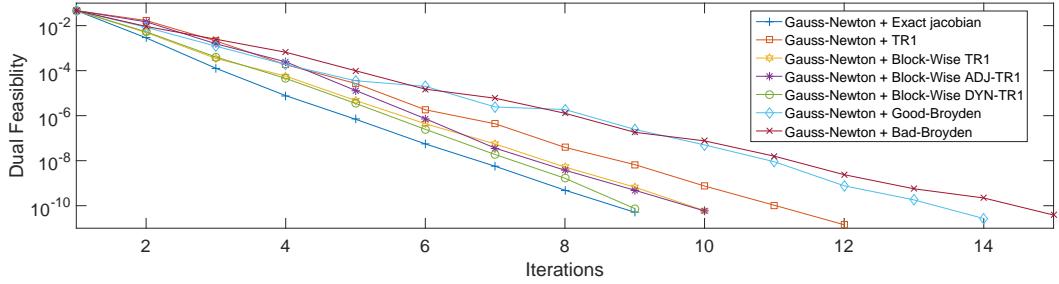}
	\caption{Local convergence analysis: comparison between different variants of the inexact adjoint-based SQP method as described in Algorithm~\ref{alg:lifted_block_TR1} based on lifted collocation, using either the exact Jacobian or different quasi-Newton type Jacobian update formulas for the nonlinear chain of 6 masses. }
	\label{figtime:one}
\end{figure*}

\subsubsection{Computational timing results for block-TR1 based lifted collocation}

\begin{figure}[tpb]
	\vspace{-1mm}
	\subfigure{\includegraphics[trim={0.4in 0in 0.3in 0in},clip,scale=0.33]{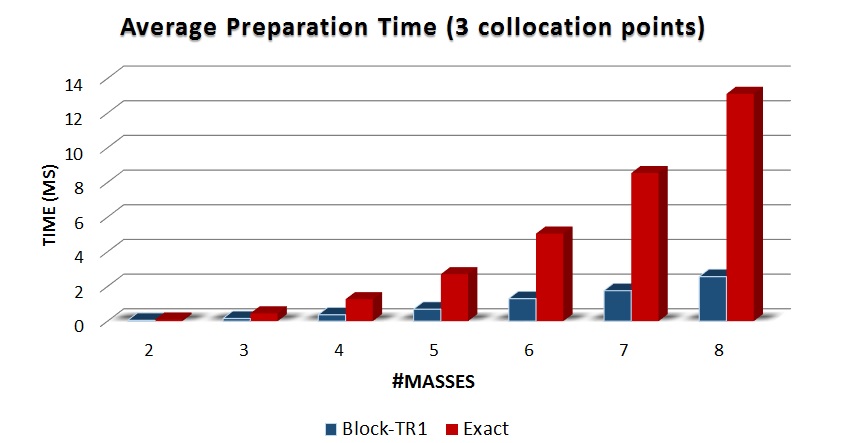}}
	\subfigure{\includegraphics[trim={0.4in 0in 0.3in 0in},clip,scale=0.33]{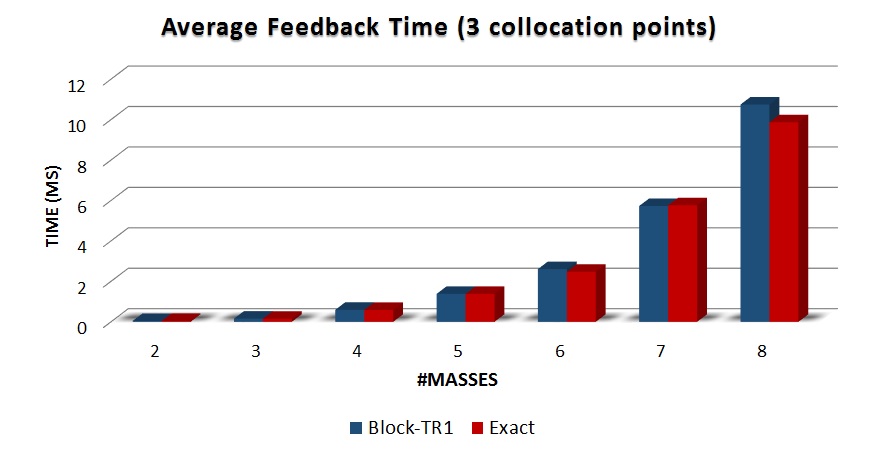}}
	\vspace{-5mm}
	\caption[blabla]{\small{Comparison of the average preparation and feedback computation times~(in ms): block-TR1 versus exact Jacobian based lifted collocation SQP method.} \footnotemark} 
	\label{figtime:two}
\end{figure}

Figure~\ref{figtime:two} illustrates the computation times of both the preparation and feedback steps of an NMPC implementation for a chain of $\nm = 2, \ldots, 8$ masses, using the lifted collocation based SQP method in Algorithm~\ref{alg:lifted_block_TR1}. It can be observed that the preparation time scales quadratically with the number of states for the block-TR1 implementation, instead of the cubic computational complexity when using the exact Jacobian. More specifically, the Jacobian evaluation, the factorization and matrix-matrix multiplications are replaced by adjoint differentiation sweeps and matrix-vector operations in Algorithm~\ref{alg:lifted_block_TR1}. On the other hand, the feedback time remains essentially the same because, after the linearization and QP preparation, both approaches lead to the solution of a similarly structured QP in Eq.~\eqref{SQP-MS} or~\eqref{SQP-LC}.


\footnotetext{The computation times in Figure~\ref{figtime:two} have been obtained using an Intel~i7-7700k processor @~4.20~GHz on Windows~10 with 64~GB of RAM.}


\begin{table}[tpb]
	\setlength{\tabcolsep}{0.8em}
	\tbl{Average computation times~(in ms) for nonlinear MPC on a chain of $\nm=6$ masses, i.e., $30$ differential states~($4$ Gauss collocation nodes versus $10$ steps of RK4).}
	{\begin{tabular}{l c c c c c c  }
		\toprule
		\multicolumn{1}{c}{ }   	& \multicolumn{3}{c}{Explicit~(RK4 in Alg.~\ref{alg:block_TR1})}   	& \multicolumn{3}{c}{Implicit~(GL4 in Alg.~\ref{alg:lifted_block_TR1})}     	      \\
		\cmidrule{2-3} \cmidrule{5-6}
		& exact & block-TR1 &			& exact & block-TR1 &						 \\
		\midrule
		Linearization     & $32.36$  &  $5.33$  &  $\mathbf{16}${\bf\%}  &  $291.37$  &  $35.99$  &  $\mathbf{12}${\bf\%} \\
		QP solution       & $23.22$  &  $37.82$  &         &  $26.33$  &  $27.86$ &           \\
		\midrule
		Total RTI step    & $56.39$  &  $43.99$  &  $78\%$  &  $318.58$  &  $64.69$  &  $20\%$ \\
		\bottomrule
	\end{tabular}}
\label{tab:chain_results}
\end{table}

Table~\ref{tab:chain_results} provides a more detailed comparison between the exact Jacobian and the proposed block-TR1 variant of the real-time iterations for NMPC, using an ARM Cortex-A53 processor. The table shows these results for both the explicit Runge-Kutta method of order~4~(RK4) in combination with Algorithm~\ref{alg:block_TR1} and using the implicit 4-stage Gauss-Legendre~(GL4) method within Algorithm~\ref{alg:lifted_block_TR1}. The proposed block-TR1 algorithm results in a computational speedup of about factor $6-8$ for the problem linearization step. In order to obtain a relatively fair comparison, the number of integration steps for RK4 has been chosen such that the numerical accuracy is close to that of the 4-stage GL method. However, since the system dynamics for the chain of masses are non-stiff, an explicit integration scheme should instead typically perform better in terms of computational efficiency.

\subsection{Nonlinear MPC for vehicle control on a snow-covered road}

\begin{figure*}
\centering
\includegraphics[trim={1.2in 0in 0.3in 0in},clip,scale=0.33]{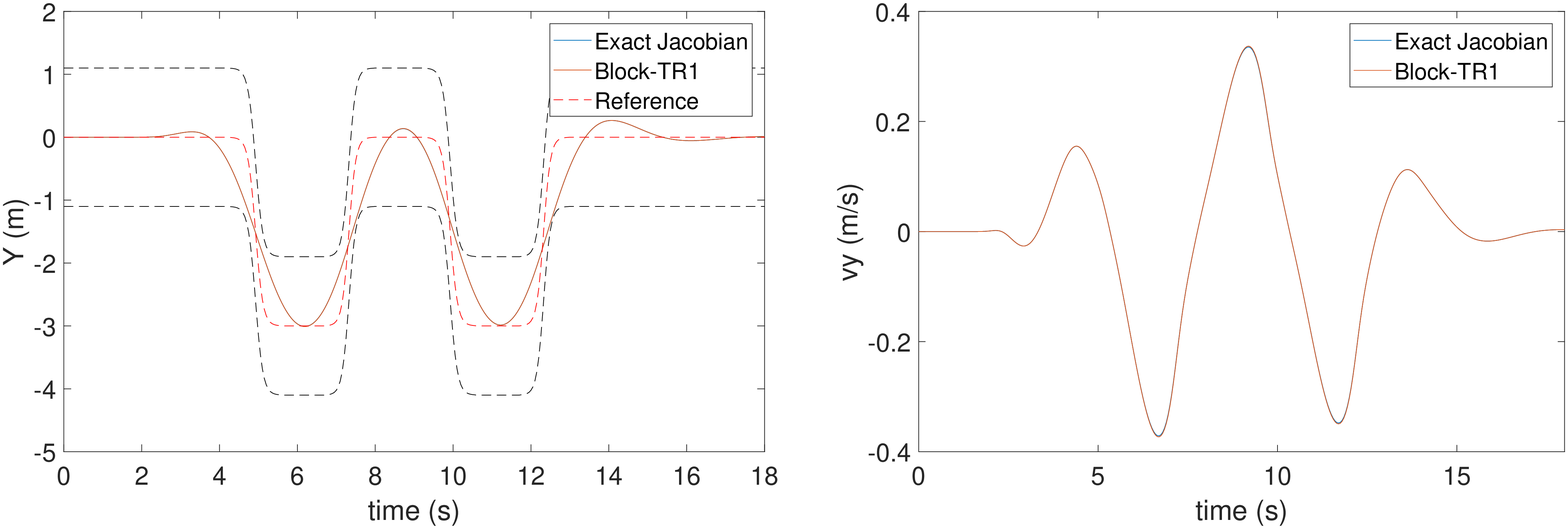}
\caption{Closed-loop NMPC performance of two double lane changes at a vehicle speed of $10$~m/s on snow-covered road conditions, using model parameters from~\cite{Berntorp2014a} in the nonlinear OCP formulation~\cite{steeringNMPC}.}
\label{figtime:four}
\end{figure*}


Our second case study considers nonlinear MPC for real-time vehicle control as motivated by automotive applications in autonomous driving. The nonlinear optimal control problem formulation is based on single-track vehicle dynamics with a Pacejka-type tire model~\cite{steeringNMPC}. The experimentally validated model parameters can be found in~\cite{Berntorp2014a}. As often the case in practice, these vehicle dynamics are rather stiff such that an implicit integration scheme should preferably be used. Therefore, it forms an ideal case study for the proposed lifted collocation based RTI method of Algorithm~\ref{alg:lifted_block_TR1}. Let us perform the closed-loop NMPC simulations as presented in~\cite{steeringNMPC}, but using the proposed block-TR1 based RTI implementation. We carried out numerical simulations for two successive double lane changes on snow-covered road conditions. The resulting closed-loop trajectories for both the exact Jacobian and the block-TR1 method are indistinguishable from each other, as illustrated in Figure~\ref{figtime:four}.


\begin{table}[tpb]
	\setlength{\tabcolsep}{0.8em}
	\tbl{Average computation times~(in ms) for vehicle control based on a single-track vehicle model within NMPC~($4$ Gauss collocation nodes versus $30$ steps of RK4).}
	{\begin{tabular}{l  c c c  c c c  }
			\toprule
			\multicolumn{1}{c}{ }   	& \multicolumn{3}{c}{Explicit~(RK4 in Alg.~\ref{alg:block_TR1})}   	& \multicolumn{3}{c}{Implicit~(GL4 in Alg.~\ref{alg:lifted_block_TR1})}     	      \\
			\cmidrule{2-3} \cmidrule{5-6}
			& exact & block-TR1 &			& exact & block-TR1 &						 \\
			\midrule
			Linearization   & $106.73$  &  $75.78$  &  $\mathbf{71}${\bf\%}  &  $52.22$  &  $18.27$  &  $\mathbf{35}${\bf\%} \\
			QP solution     & $4.46$  &  $4.51$  &    &  $4.59$  &  $4.72$  &   \\
			\midrule
			Total RTI step  & $111.79$  &  $80.94$  &  $72\%$  &  $57.43$  &  $23.64$  &  $41\%$ \\
			\bottomrule
	\end{tabular}}
	\label{tab:steering_results}
\end{table}
 
The corresponding computation times on the ARM Cortex-A53 processor are illustrated in detail by Table~\ref{tab:steering_results}. Because of the relatively stiff system dynamics, the proposed block-TR1 lifted collocation method from Algorithm~\ref{alg:lifted_block_TR1} becomes attractive and additionally provides a computational speedup of about factor $3$ over the standard exact Jacobian based implementation. Note that, even though the Raspberry Pi~3 is not an embedded processor by itself, it uses an ARM core of the same type as those that are used by multiple high-end automotive microprocessors. Therefore, the proposed algorithm implementation as well as the corresponding numerical results form a motivation for real-time embedded control applications that involve a relatively large, implicit and/or stiff system of differential equations.


\section{Conclusions and outlook}
\label{sec:concl}

In this paper, we proposed a block-wise sparsity preserving two-sided rank-one~(TR1) Jacobian update for an adjoint-based inexact SQP method to efficiently solve the nonlinear optimal control problems arising in NMPC. We proved local convergence for the block-structured quasi-Newton type Jacobian matrix updates. In case of a Gauss-Newton based SQP implementation, we additionally showed that the asymptotic rate of contraction remains the same. We also presented how this approach can be implemented efficiently in a tailored lifted collocation framework, in order to avoid matrix factorizations and matrix-matrix multiplications. Finally, we illustrated the local convergence properties as well as the computational complexity results numerically for two nonlinear MPC case studies.
The effect of the presented contraction properties on the convergence and closed-loop stability of the block-TR1 based real-time iterations is an important topic that is part of ongoing research.

\bibliographystyle{tfs}
\bibliography{IEEEabrv,TR1b}

\end{document}